\documentclass[a4paper,11pt,reqno]{amsart} 
\usepackage{amssymb,amsthm,amsmath,amsxtra,dsfont,bbm,stix,color}
\usepackage{hyperref} 

\numberwithin{equation}{section}
\makeatletter

\makeatletter
\def\@tocline#1#2#3#4#5#6#7{\relax
  \ifnum #1>\c@tocdepth 
  \else
    \par \addpenalty\@secpenalty\addvspace{#2}%
    \begingroup \hyphenpenalty\@M
    \@ifempty{#4}{%
      \@tempdima\csname r@tocindent\number#1\endcsname\relax
    }{%
      \@tempdima#4\relax
    }%
    \parindent\z@ \leftskip#3\relax \advance\leftskip\@tempdima\relax
    \rightskip\@pnumwidth plus4em \parfillskip-\@pnumwidth
    #5\leavevmode\hskip-\@tempdima
      \ifcase #1
       \or\or \hskip 1em \or \hskip 2em \else \hskip 3em \fi%
      #6\nobreak\relax
      \dotfill
      \hbox to\@pnumwidth{\@tocpagenum{#7}}
    \par
    \nobreak
    \endgroup
  \fi}
\makeatother

\theoremstyle{plain}
\newtheorem{theorem}{Theorem}
\newtheorem{lemma}[theorem]{Lemma}
\newtheorem{corollary}[theorem]{Corollary}
\newtheorem{proposition}[theorem]{Proposition}

\theoremstyle{definition}
\newtheorem{definition}[theorem]{Definition}
\newtheorem{assumption}[theorem]{Assumption}

\theoremstyle{remark}

\newcommand{\bra}[1]{\left\langle #1\right|}
\newcommand{\ket}[1]{\left|#1\right\rangle}
\newcommand{\wto}{\rightharpoonup}
\DeclareMathOperator{\Tr}{Tr}

\newcommand{\nor}{\mathrm{nor}}

\newcommand{\cA}{\mathcal{A}}
\newcommand{\cL}{\mathcal{L}}
\newcommand{\cS}{\mathcal{S}}
\newcommand{\bN}{\mathbb{N}}
\newcommand{\R}{\mathbb{R}}

\newcommand{\cN}{\mathcal{N}}

\newcommand{\gH}{\mathfrak{H}}
\newcommand{\gF}{\mathfrak{F}}
\newcommand{\gh}{\mathfrak{h}}
\newcommand{\cB}{\mathcal{B}}
\newcommand{\cE}{\mathcal{E}}

\newcommand{\cP}{\mathcal{P}}
\newcommand{\cM}{\mathcal{M}}

\newcommand{\1}{\mathds{1}}
\newcommand{\eps}{\varepsilon}
\newcommand{\ada}{a ^{\dagger}}
\newcommand{\cda}{c ^{\dagger}}
\newcommand{\im}{\mathrm{i}}

\newcommand{\Gammat}{\widetilde{\Gamma}}

\newcommand{\gHt}{\mathfrak{H}_{\rm tot}}
\newcommand{\cvnet}{\overset{\star}{\underset{\mathrm{net}}{\rightharpoonup}}}

\newcommand{\cEP}{\cE_{\mathrm{Pek}}}
\newcommand{\id}{\mathds{1}}

\numberwithin{theorem}{section}
\numberwithin{equation}{section}

\begin{document}

\title{Convergence of states for polaron models in the classical limit}

\author[M. Falconi]{Marco Falconi}
\address{Department of Mathematics, Politecnico di Milano, Piazza Leonardo da Vinci 32, 20133 Milano}
\email{marco.falconi@polimi.it}

\author[A. Olgiati]{Alessandro Olgiati}
\address{Department of Mathematics, Politecnico di Milano, Piazza Leonardo da Vinci 32, 20133 Milano}
\email{alessandro.olgiati@polimi.it}

\author[N. Rougerie]{Nicolas Rougerie}
\address{Ecole Normale Sup\'erieure de Lyon \& CNRS,  UMPA (UMR 5669)}
\email{nicolas.rougerie@ens-lyon.fr}

\date{January 2025}

\maketitle

\begin{abstract}
We consider the quasi-classical limit of Nelson-type regularized polaron models describing a particle interacting with a quantized bosonic field. We break translation-invariance by adding an attractive external potential decaying at infinity, acting on the particle. In the strong coupling limit where the field behaves classically we prove that the model's energy quasi-minimizers strongly converge to ground states of the limiting Pekar-like non-linear model. This holds for arbitrarily small external attractive potentials, hence this binding is fully due to the interaction with the bosonic field. We use a new approach to the construction of quasi-classical measures to revisit energy convergence, and a localization method in a concentration-compactness type argument to obtain convergence of states.
\end{abstract}

\tableofcontents

\section{Introduction}

A quantum particle interacting with a quantized bosonic field (e.g. an electron interacting with the phonons of a crystal) may exhibit self-trapping: the particle is confined in a ``hole'' of its own making in the field. Usual linear models (Fr\"ohlich's polaron, Nelson model) are translation invariant and this phenomenon thus may not take the form of existence of actual bound states. One of the strongest mathematical evidence for the phenomenon is the existence of energy minimizers for the non-linear quasi-classical approximations of the models. In the case of the Fr\"ohlich polaron~\cite{Moller-06,Seiringer-19}, the limiting model is Pekar's, for which the existence and uniqueness (up to translations) of ground states was proven in~\cite{Lieb-77,Lions-80}. Other related models also exhibit this phenomenon, some being studied e.g. in~\cite{BreFauPay-22a,BreFauPay-22b,FroJonLen-07,FroJonLen-07b,LewRou-13,LewRou-13b,Ricaud-14}.    

The validity of the quasi-classical approximation has been established in~\cite{DonVar-83,LieTho-97,MiySpo-07} in the strong-coupling limit at the level of the ground state energy. Quantum corrections are investigated in~\cite{BroSei-22,FraSei-19,FelSei-21}. The corresponding dynamical problem is considered e.g. in~\cite{CorFalOli-19,FraGan-17,LeoRadSchSei-19,LeoMitSei-20,FalLeoMitPet-21,GriSchSch-17,Griesemer-17}. If a trapping external potential (increasing to infinity at spatial infinity) is further added to the model, the convergence of ground energy states to quasi-classical minimizers is proved in~\cite{CorFalOli-20}. 

Here we shall break the translation invariance by an arbitrarily small, decaying, external attractive potential, and prove that this is sufficient for self-trapping in the quasi-classical limit. For simplicity we consider Nelson-type models with regular particle-field interactions, where the definition of the Hamiltonian 
\begin{equation}\label{eq:intro-hamil}
	\begin{split}
		H_\alpha ^{(V)} := \;&\left(-\Delta + V\right) \otimes \1 \\
		&+ \alpha^{-2}\left( \1 \otimes \int_{\R^d} T(k) \widehat{c}^\dagger (k) \widehat{c} (k)\,dk 
		+ \alpha \int_{\R^d} \left( e^{\im k\cdot x} \widehat{v} (k) \widehat{c}^{\dagger} (k) + \mathrm{h.c.} \right)dk  \right)
	\end{split}
\end{equation}
as a self-adjoint operator is straightforward (the Lieb-Yamakazi method~\cite{LieYam-58} or Gross transformation~\cite{Seiringer-19} are not needed). The above acts on
\begin{equation}\label{eq:intro-hilbert}
\gH := L^2 (\R^d) \otimes \gF (L^2 (\R^d)), 
\end{equation}
the tensor product of the particle and field Hilbert spaces, the latter being the bosonic Fock space constructed from the one-particle space $L^2 (\R^d)$,
\begin{equation*}\gF (L^2 (\R^d)) = \bigoplus_{n\geq 0} \left(L^2 (\R^d)\right)^{\otimes_{\mathrm{sym}} n}.
\end{equation*}
The particle's coordinate is labeled by $x$, i.e. $e^{\im k\cdot x}$ acts as mutliplication on the particle's side. The standard bosonic operators $\widehat{c}^\dagger (k),\widehat{c}(k)$ create/annihiliate a field excitation in the Fourier mode $k\in\R^d$, and satisfy usual canonical commutation relations (CCR). 

The limit $\alpha \to \infty$ is a strong coupling one. A heuristic square completion in the second term of~\eqref{eq:intro-hamil} indicates that the number of field excitations is of order $\alpha^2$ in this limit. To obtain a well-defined limit we therefore multiply all terms involving the field degrees of freedom in~\eqref{eq:intro-hamil} by $\alpha^{-2}$. For the true Fr\"ohlich polaron model, this is equivalent to a change of length/energy units~\cite{Seiringer-19}. 

One can see the strong coupling regime as a quasi-classical limit (i.e. a semi-classical limit for the field degrees of freedom only) by redefining creators/annihilators in the manner
\begin{equation}\label{eq:scaleop}
\widehat{a}^\dagger (k) := \alpha^{-1} \widehat{c}^\dagger (k), \quad \widehat{a} (k) := \alpha^{-1} \widehat{c} (k) 
\end{equation}
so that 
\begin{equation}\label{eq:intro-hamil-bis}
H_\alpha ^{(V)} := \left(-\Delta + V\right) \otimes \1 + \1 \otimes \int_{\R^d} T(k) \widehat{a}^\dagger (k) \widehat{a} (k) + \int_{\R^d} \left( e^{\im k\cdot x} \widehat{v} (k) \widehat{a}^{\dagger} (k) + \mathrm{h.c.} \right)  
\end{equation}
and
\begin{equation}\label{eq:intro CCR}
[ \widehat{a} (k), \widehat{a}^\dagger (k')] = \alpha^{-2} \delta_{k=k'}, \quad  [ \widehat{a} (k), \widehat{a} (k')] = 0, \quad [ \widehat{a}^\dagger (k), \widehat{a}^\dagger (k')] = 0
\end{equation}
for all $k,k'\in \R^d$. The data of the problem are 
\begin{itemize}
 \item The field's dispersion relation $T:\R^d \mapsto \R^+$ for which we assume a gap at $0$ (i.e. $T(k) \geq c >0$ for all $k$), to avoid infrared problems. In polaron models one typically takes $T \equiv 1$. 
 \item The field-particle interaction potential $v\in L^2 (\R^d)$. For the Fr\"ohlich polaron one should consider a singular dipole-charge interaction, something we could include with extra effort. 
 \item The external potential $V:\R^d \mapsto \R^-$. Our point is that it can be arbitrarily small (but negative), so that we impose $V(x) \underset{|x| \to \infty}{\to} 0$. 
\end{itemize}
We give more precise definitions and assumptions below. Our main goal is to show that binding holds in the $\alpha\to \infty$ limit for arbitrary $V$ \emph{decaying at infinity}, thus generalizing results of~\cite{CorFalOli-20} applying to trapping potentials. We opt to consider only the case $v\in L^2(\R^d)$ because the difficulties linked to singular $v$ on the one hand, and to lack of trapping on the other hand are rather orthogonal.

Our analysis bears on sequences of approximate ground states for $H_\alpha^{(V)}$, whose energy reproduce the infimum of the spectrum up to small corrections in the limit $\alpha \to \infty$. Let a sequence $(\Psi_\alpha)_\alpha \in \gH$ be such that 
\begin{equation}\label{eq:approx GS}
\langle \Psi_\alpha | H_\alpha^{(V)} \Psi_\alpha\rangle_{\gH} \leq \inf \sigma (H_\alpha^{(V)}) +o_{\alpha} (1), \quad \left\Vert\Psi_\alpha \right\Vert_\gH = 1. 
\end{equation}
In particular we can take for each $\alpha$ a sequence $(\Psi_{n,\alpha})_n$  such that  
\begin{equation*} \langle \Psi_{n,\alpha} | H_\alpha^{(V)} \Psi_{n,\alpha} \rangle_{\gH} \underset{n\to \infty}{\to} \inf \sigma (H_\alpha^{(V)})\end{equation*}
and diagonally extract a subsequence $(\Psi_{n,\alpha(n)})_n$. A consequence of our main results below is that if 
\begin{equation}\label{eq:intro DM}
\gamma_\alpha := \Tr_{\gF(L^2 (\R^d))} \left( |\Psi_\alpha \rangle \langle \Psi_\alpha|\right)  
\end{equation}
is the particle's (reduced) density matrix then 
\begin{equation}\label{eq:intro main 1}
\boxed{\gamma_\alpha \underset{\alpha \to \infty}{\to} \int |u\rangle \langle u| dP (u) \;\;\mbox{ strongly in trace-class norm}}
\end{equation}
where $P$ is a Borel probability measure on the set of minimizers of the quasi-classical Pekar functional obtained as 
\begin{equation*} \cEP^{(V)} [u] := \min_{\varphi \in L^2 (\R^d)} \langle u \otimes \xi (\varphi) | H_\alpha^{(V)} u \otimes \xi (\varphi)\rangle_\gH.\end{equation*}
Here
\begin{equation*} \xi (\varphi) = e^{a^\dagger (\varphi) -a (\varphi) } 1 \oplus 0 \oplus 0 \oplus \ldots \in \gF (L^2 (\R^d))\end{equation*}
is a coherent state of the field (the definition of $a(\varphi),a^\dagger (\varphi)$ is recalled in~\eqref{eq:annihi} below). The above means that an arbitrarily small potential well is sufficient to trap/bind the particle in the quasi-classical limit. 
%

\medskip

{\small\noindent\textbf{Acknowledgments:} Funding from the European Research Council
(ERC) under the European Union's Horizon 2020 Research and Innovation
Programme (Grant agreement CORFRONMAT No 758620) is gratefully
acknowledged. M.F.\ also acknowledges the support of the MUR grant
``Dipartimento di Eccellenza 2023-2027'', and of the ``Centro Nazionale di
ricerca in HPC, Big Data and Quantum Computing''.}

\section{Main results}\label{sec:main}

\subsection{Model}

The natural Hilbert space for a system composed by one particle in $\mathbb{R}^d$ and a quantum bosonic field is
\begin{equation}
\mathfrak{H}=L^2_\mathrm{part}(\mathbb{R}^d)\otimes \mathfrak{F},
\end{equation}
with the bosonic Fock space
\begin{equation} \label{eq:F_alpha}
\mathfrak{F}:= \mathfrak{F}\big(L^2_\mathrm{field}(\mathbb{R}^d)\big)=\bigoplus_{n=0}^\infty \left( L^2_\mathrm{field}(\mathbb{R}^d)\right)^{\otimes_\mathrm{sym}n}.
\end{equation}
In most of the sequel, the $\alpha$ dependence of the model will be encoded into the fact that the annihilation operator $a(f)$ acting on $\gF$ as \begin{equation}\label{eq:annihi}
 \left(a(f) \Psi_n \right) (x_1,\ldots,x_{n-1})= \alpha^{-1}\sqrt{n} \int_{\R^d} \overline{f(x)} \Psi_n (x,x_1,\ldots,x_{n-1})dx, \; \forall\, \Psi_n \in L^2_{\rm sym} (\R^{dn}) 
\end{equation}
and its adjoint $a^\dagger(f)$ satisfy the rescaled Canonical Commutation Relations (CCR)
\begin{equation} \label{eq:CCR}
[ a(f),a(g)]=[a^\dagger(f),a^\dagger(g)]=0,\qquad[a(f),a^\dagger(g)]=\frac{\langle f,g\rangle}{\alpha^2}, \qquad \forall f,g\in L_\mathrm{field}^2(\mathbb{R}^d).
\end{equation}

We still denote by $a(f),a^\dagger(g)$ the operators on $\mathfrak{H}$ which act as~\eqref{eq:CCR} on the factor $\mathfrak{F}$ and as the identity on the factor $L^2_\mathrm{part}(\mathbb{R}^d)$. We associate to these operators the operator-valued distributions $a_x,a^\dagger_x$ defined by
\begin{equation}
a(f)=\int_{\R^d} \overline{f(x)} a_x\, dx,\qquad a^\dagger(f)=\int_{\R^d} f(x) a^\dagger_x\, dx,
\end{equation}
together with their Fourier transforms
\begin{equation}
\widehat{a}_k= \frac{1}{(2\pi)^{d/2}}\int_{\R^d} e^{ik\cdot x} a_x\, dx,\qquad \widehat{a}^\dagger_k=\frac{1}{(2\pi)^{d/2}} \int_{\R^d} e^{-ik\cdot x}a^\dagger_x\,dx
\end{equation}
and the number operator
\begin{equation}\label{eq:scaled number}
\cN_\alpha := \int_{\R^d} \ada_x a_x dx = \int_{\R^d} \widehat{a}^\dagger_k \widehat{a}_k dk.
\end{equation}
Our Hamiltonian acts on $\mathfrak{H}$ as
\begin{equation} \label{eq:H_alpha^V}
\begin{split}
H^{(V)}_\alpha=\;&\left(-\Delta_x+ V(x)\right)\otimes \1+ \1\otimes \cN_\alpha + \int_{\R^d} \left(e^{ik\cdot x}\widehat v(k)\widehat {a}^\dagger_k+e^{-ik\cdot x}\overline{ \widehat {v}(k)}\widehat a_k\right)dk\\
=\;&\left(-\Delta_x+ V(x)\right)\otimes \1+ \1\otimes \cN_\alpha + \int_{\R^d} \left(v(x-y)a^\dagger_y+\overline{v(x-y)}a_y\right)dy
\end{split}
\end{equation}
We made the simplifying choice $T \equiv 1$ in~\eqref{eq:intro-hamil}, corresponding to the Fr\"ohlich polaron. We can easily accomodate functions such as 
\begin{equation*}T(k) = \left(|k|^2 + 1\right)^s, 0\leq s \leq 1\end{equation*}
i.e. a kinetic energy operator for the field such as $\left( 1 -\Delta \right)^s$. This requires only an extra use of an IMS-like localization formula (from~\cite[Appendix~B]{LenLew-11} for $s\neq 0,1$) in Lemma~\ref{lemma:field_localization} below.

Here $V$ is an external potential, $v$ is a function which couples the field and particle modes, and $\widehat{v}$ is its Fourier transform.
\begin{assumption}[\textbf{The external potential}]\mbox{} \label{assum:V}\\
	We assume that $V\in L^\infty_{\mathrm{part}}(\mathbb{R}^d)$ is a strictly negative function such that 
	\begin{equation}
	\lim_{|x|\to\infty} V(x)=0.
	\end{equation}
\end{assumption}

\begin{assumption}[\textbf{The interaction}]\mbox{}\label{assum:int}\\ 
	We assume that $v\in L^2(\mathbb{R}^d)$ is real-valued.
\end{assumption}

We define the ground state energy
\begin{equation}
E^{(V)}_\alpha=\inf \sigma (H^{(V)}_\alpha).
\end{equation}
It is known~\cite{DonVar-83,LieTho-97,MiySpo-07} that, to leading order as $\alpha\to\infty$, $E^{(V)}_\alpha$ is close to the minimal energy obtained through product trial states of the form 
\begin{equation*}\Psi=\psi\otimes \xi(u),\end{equation*} 
where $\psi\in L^2_\mathrm{part}(\mathbb{R}^d)$ is a particle wave-function and $\xi(u)\in \mathfrak{F}$ is the field coherent state defined by
\begin{equation} \label{eq:coherent_state}
	\xi(u)=e^{a^\dagger(u)-a(u)}\Omega \in \mathfrak{F}
\end{equation}
for $u\in L^2_\mathrm{field}(\mathbb{R}^d)$. Here 
\begin{equation*}\Omega=1\oplus0\oplus0\oplus\dots\in\mathfrak{F}\end{equation*}
is the vacuum vector.

The expectation of $H_\alpha^{(V)}$ in $\Psi=\psi\otimes \xi(u)$ reads
\begin{equation} \label{eq:qc_energy}
	\begin{split}
		\big\langle \psi\otimes \xi(u), H_\alpha^{(V)}\, \psi\otimes \xi(u)\big \rangle =\;& \int_{\mathbb{R}^d} |\nabla \psi(x)|^2dx+ \int_{\mathbb{R}^d} V(x)|\psi(x)|^2dx\\
		&+\|u\|_2^2+ \int_{\mathbb{R}^d\times \mathbb{R}^d}\left(u(y)+\overline{u(y)}\right)v(x-y)|\psi(x)|^2dxdy
	\end{split}
\end{equation}
Minimizing the above expression with respect to $u$ (which is tantamount to a square completion) yields\footnote{For $T= \left(|k|^2 + 1\right) ^s$ this is replaced by $u_\psi = -(1-\Delta)^{-s}\left(|\psi|^2 \ast v\right)$.} 
\begin{equation}\label{eq:fieldconfig}
u = u_\psi := - |\psi|^2 \ast v
\end{equation}
and thus the Pekar functional is 
\begin{equation} \label{eq:pekar_functional}
\begin{split}
\mathcal{E}_\mathrm{Pek}^{(V)}(\psi)=\;&\int_{\mathbb{R}^d}|\nabla\psi(x)|^2dx+ \int_{\mathbb{R}^d}V(x)|\psi(x)|^2dx\\
&-\iiint_{\mathbb{R}^d\times\mathbb{R}^d\times \mathbb{R}^d}v(x-y)v(x-z)\left|  \psi(y)\right|^2|\psi(z)|^2dxdydz.
\end{split}
\end{equation}
Observe that 
\begin{equation*} \iiint_{\mathbb{R}^d\times\mathbb{R}^d\times \mathbb{R}^d}v(x-y)v(x-z)\left|  \psi(y)\right|^2|\psi(z)|^2dxdydz = \iint_{\mathbb{R}^d\times\mathbb{R}^d} W(x-y) \left|  \psi(x)\right|^2|\psi(y)|^2 dxdy\end{equation*}
with 
\begin{equation}\label{eq:eff pot}
W(x-y) = \iint_{\mathbb{R}^d\times\mathbb{R}^d} v(z-x)v(z-y) dz = \iint_{\mathbb{R}^d\times\mathbb{R}^d} v(z)v(z+x-y) dz = \left(v\ast v\right) (y-x) 
\end{equation}
so that the effective interaction term in~\eqref{eq:pekar_functional} is always a non-positive/attractive pair interaction, in the sense that 
\begin{equation*}\widehat{W} (k) = \left|\widehat{v} (k)\right|^2 \geq 0.\end{equation*}
Also note that, since we assume $v\in L^2$ we have that $W\in L^\infty$ by Young's inequality.

\subsection{Statements}

We define the minimal Pekar energy at mass $m >0$ in the manner
\begin{equation}\label{eq:pekar_energy}
E_\mathrm{Pek}^{(V)}(m)=\inf\left\{ \mathcal{E}^{(V)}_\mathrm{Pek}(\psi)\;|\;\|\psi\|_{2}^2=m \right\} = \mathcal{E}_\mathrm{Pek}^{(V)}\left( \psi_{V,m} \right),
\end{equation}
with the convention that $E_\mathrm{Pek}^{(0)}(m)$ is the minimal translation-invariant Pekar energy corresponding to the choice $V=0$. We denote 
\begin{equation}\label{eq:minimizers}
\cM _\mathrm{Pek}^{(V)}(m) = \left\{\psi\in L^2 (\R^d)\;|\; \|\psi\|_{2}^2=m,\, \mathcal{E}_\mathrm{Pek}^{(V)}\left( \psi \right) = E_\mathrm{Pek}^{(V)}(m) \right\}.
\end{equation}
That the above is not empty  (i.e. that Pekar minimizers always exist) follows from the usual concentration-compactness method as in~\cite{LewRou-13b,FroJonLen-07b,Ricaud-14} for example, or by using rearrangement inequalities as in~\cite{Lieb-77} if $W$ is assumed radial.

It follows from the results/methods of~\cite{LieTho-97,CorFalOli-20,DonVar-83} that 
\begin{equation} \label{eq:energy_convergence}
\lim_{\alpha\to\infty} E^{(V)}_\alpha =E^{(V)}_\mathrm{Pek}(1).
\end{equation}
We shall revisit a proof of~\eqref{eq:energy_convergence} along the lines of~\cite{CorFalOli-20} for completeness, providing in particular an alternative construction of the quasi-classical measures used as main tools.

In this paper we are particularly interested in the associated convergence of states, which is our main result:

\begin{theorem}[\textbf{Convergence of states in the quasi-classical limit}]\label{thm:main}\mbox{}\\
Let $\Psi_\alpha \in\mathfrak{H}$ be a (family of) normalized vector(s) such that
\begin{equation} \label{eq:minimizing_sequence}
\langle \Psi_\alpha , H^{(V)}_\alpha\Psi_\alpha\rangle\le E^{(V)}_\alpha+o_\alpha(1).
\end{equation}
Let $k,\ell$ be non-negative integers with $k+\ell \leq 2$. Modulo extraction of a subsequence in $\alpha$, for every bounded $A\in \mathcal{B}\big( L^2_\mathrm{part}(\mathbb{R}^d) \big)$ and for every $f_1,\dots,f_k,g_1,\dots,g_\ell\in L^2_\mathrm{field}(\mathbb{R}^d)$,
\begin{multline}\label{eq:mainCV}
	\sqrt{k!\ell!}\, \Big\langle \Psi_\alpha, A \otimes a^\dagger(g_1)\dots a^\dagger(g_\ell) a(f_1)\dots a(f_k) \Psi_\alpha \Big \rangle \underset{\alpha \to \infty}{\to} \\ \int_{\psi \in \cM _\mathrm{Pek}^{(V)}(1)} \langle \psi, A \psi \rangle_{L^2}  \prod_{j=1}^k\langle f_j,u_\psi \rangle \prod_{j=1}^\ell\langle u_\psi, g_j\rangle\,dP(\psi),
\end{multline}
where $P$ is a probability measure over the set of Pekar minimizers $\cM _\mathrm{Pek}^{(V)}(1)$ at mass $1$ and $u_\psi$ is defined as in~\eqref{eq:fieldconfig}.
%
\end{theorem}

A few comments:
\begin{enumerate}
 \item Picking $k=\ell=0$ in~\eqref{eq:mainCV} gives 
 \begin{equation*} \Tr\left[\gamma_\alpha A \right] \underset{\alpha \to \infty}{\to} \int_{\psi \in \cM _\mathrm{Pek}^{(V)}(1)} \langle \psi, A \psi \rangle_{L^2} dP (\psi)\end{equation*}
 for any bounded operator $A$, where the particle reduced density matrix is defined as in~\eqref{eq:intro DM}. Hence 
 \begin{equation*} \gamma_\alpha \overset{\mathrm{*}}{\wto} \int_{\psi \in \cM _\mathrm{Pek}^{(V)}(1)} |
   \psi \rangle \langle \psi | dP (\psi)
   \end{equation*} weakly-star in the trace-class. Since $P$ is a
 probability, the right-hand side has trace $1$. Hence the trace of $\gamma_\alpha$
 converges. The latter equals the trace-class norm because $\gamma_\alpha \geq 0$. The
 convergence is thus strong in the
 trace-class (see~\cite{dellAntonio-67} or~\cite[Addendum~H]{Simon-79}), as claimed in~\eqref{eq:intro
   main 1}.
 \item Taking $A=\id$, $k=\ell = 1$ in~\eqref{eq:mainCV} and varying $f_1,g_1 \in
   L^2 (\R^d)$ yields that the field reduced density matrix described by the
   integral kernel
\begin{equation*} \Tr_{L^2_{\rm part} (\R^d)} \left[ | \Psi_\alpha \rangle \langle \Psi_\alpha | a^{\dagger}_xa_y\right]\end{equation*} 
converges weakly-$*$ in the trace-class to the operator with kernel
\begin{equation*}  \int_{\psi \in \cM _\mathrm{Pek}^{(V)}(1)} \overline{u_{\psi}(x)}u_{\psi}(y)  dP (\psi).\end{equation*}
Taking $k=1$ or $\ell = 1$, varying $A,f_1,g_1$, gives a similar convergence for what we shall call the field-particle density matrix in Section~\ref{sect:loc}. 
 \item The limitation $k+\ell \leq 2$ comes from the fact that, for general quasi-minimizing sequences, we only have a control via the energy on the expectation of the field excitation number 
 \begin{equation*}\cN =\sum_{j\geq 1} \id_{L^2(\R^d)} \otimes a^\dagger (f_j) a(f_j),\end{equation*} 
 with $(f_j)_j$ a orthonormal basis of $L^2_{\rm field} (\R^d)$. Under the stronger assumption that 
 \begin{equation*} \left\langle \Psi_\alpha | \cN^\kappa \Psi_\alpha \right\rangle \leq C_\kappa\end{equation*}
 independently of $\alpha$, we can allow for $k+\ell \leq 2 \kappa$ in the main result. If
 $\Psi_\alpha$ is a true eigenstate of the Hamiltonian (assuming such exist),
 estimates of this form follow from the variational equation and so-called
 pull-through formulae~\cite{Ammari-00,Rosen-71}.
 \item Again, since the external potential can be arbitrarily small, its only function is to break translation invariance. The binding/self-trapping only comes from the particle-field interaction. In particular, in space dimensions $d\geq 3$, the Cwikel-Lieb-Rosenblum inequality~\cite[Chapter~4 and references therein]{LieSei-09} ensures that if $\| V\|_{L^{d/2} (\R^d)}$ is small enough, the Schr\"odinger operator $-\Delta + V$ acting on the particle has no bound states.
 \item Our proof does not require a purely negative external potential $V$, but only that 
 \begin{equation*} E_\mathrm{Pek}^{(V)}(1) < E_\mathrm{Pek}^{(0)}(1)\end{equation*}
 which is certainly the case for $V < 0$ by using a translation-invariant ground state as trial state for the functional with trapping potential.
 \end{enumerate}

\subsection{Organization of the paper}	

In essence we combine the philosophies of~\cite{CorFalOli-20} and~\cite{LewNamRou-14,LewNamRou-14d}: quasi-classical measures, and systematic combination thereof with localization methods. This allows  a concentration-compactness-type analysis of the many-body problem in the quasi-classical limit reminiscent of what was performed in~\cite{LewNamRou-14} for the Bose gas in the mean-field limit (see~\cite{Rougerie-spartacus,Rougerie-LMU,Rougerie-EMS} for review).

We will start by defining reduced particle/field/field-particle density matrices in Section~\ref{sect:loc}. An important tool is then to define states localized in a given region of space, such that ``the reduced densities of the localized state are the localizations of the densities of the initial state'' in the spirit of~\cite{Lewin-11} and references therein.

With this we may split the energy into the contribution of the region localized close to the potential well, and the contribution of the complement. Using~\eqref{eq:energy_convergence} for both terms separately and simple binding properties of the classical Pekar energies (essentially that the classical energy in the potential well is smaller), we conclude that it is energetically favorable to have all the mass concentrated close to the potential well, which leads to binding/strong convergence. 

We find it useful to revisit the construction of quasi-classical measures in Section~\ref{sect:deF}. This was performed in~\cite{CorFal-18,CorFalOli-19,CorFalOli-19b,CorFalOli-20} based on a Weyl quantization approach building on~\cite{AmmNie-08,AmmNie-09}. We provide an alternative construction yielding slightly stronger results by using anti-Wick quantization as in~\cite{LewNamRou-14,LewNamRou-14d}, combining  with ideas from~\cite{FanLewVer-88}. 

Using the simplified construction of the measures, we give a self-contained proof of~\eqref{eq:energy_convergence} in Section~\ref{sec:energy} for completeness, and because some of the steps are re-used when finally completing the proof of Theorem \ref{thm:main}  in Section~\ref{sec:limit}.

	\section{Reduced densities, localization of states and localization of energies} \label{sect:loc}
	
\subsection{Reduced densities}\label{sec:DMs}	\mbox{}
	
	\smallskip
	
	\noindent\textbf{Identifying states on $\mathfrak{H}=L^2_\mathrm{part}(\mathbb{R}^d)\otimes\mathfrak{F}$ with $\mathfrak{F}$-valued functions.} We recall that, by the Schmidt decomposition, any $\Psi\in \mathfrak{H}$ can be written as
	\begin{equation} \label{eq:schmidt}
	\Psi= \sum_{j=1}^\infty c_j u_j\otimes v_j,
	\end{equation}
	where $\{u_j\}_{j\in\mathbb{N}}$ is a suitable orthonormal set in $L^2_\mathrm{part}(\mathbb{R}^d)$, $\{v_j\}_{j\in\mathbb{N}}$ is a suitable orthonormal set in $\mathfrak{F}$, and $c_j\ge0$ for all $j$. By construction then $\|\Psi\|^2=\sum_{j=1}^\infty c_j^2$.
	
	We also recall that $\Psi\in\mathfrak{H}$ can be written canonically as an element of the space $L^2(\mathbb{R^d},\mathfrak{F})$. If $\Psi= u\otimes v$, where $u\in L^2_\mathrm{part}(\mathbb{R}^d)$ and $v\in\mathfrak{F}$, is a factorized state, then $\Psi$ is identified with the function $\Psi(\cdot):\mathbb{R}^d\to \mathfrak{F}$ whose action is
	\begin{equation} \label{eq:hilbert_identification}
	\Psi(x)=u(x) v.
	\end{equation}
	This definition is then extended by linearity to the whole $\mathfrak{H}$. 
	
	The identification \eqref{eq:hilbert_identification} between vectors induces a similar one for states on $\mathfrak{H}$. A pure (normal) state on $\mathfrak{H}$ is a rank-one projection $\Gamma=\ket{\Psi}\bra{\Psi}$, where $\Psi\in\mathfrak{H}$ has unit norm. By \eqref{eq:schmidt} there exists an orthonormal set $\{u_j\}_{j\in\mathbb{N}}$ of elements of $L^2_\mathrm{part}(\mathbb{R}^d)$, an orthonormal set $\{v_j\}_{j\in\mathbb{N}}$ of elements of $L^2_\mathrm{field}(\mathbb{R}^d)$, and coefficients $\{c_j\}_{j\in\mathbb{N}}$ such that
	\begin{equation} \label{eq:pure_state_decomp}
	\Gamma= \sum_{j,k=1}^\infty c_j\overline{c_k} | u_j \rangle \langle u_k | \otimes |v_j\rangle \langle v_k|.
	\end{equation}
	We naturally identify $\Gamma$ with the rank-one projection on $L^2(\mathbb{R}^d,\mathfrak{F})$ whose integral kernel is
	\begin{equation} \label{eq:Gamma_x_y}
	\Gamma(x,y)= \sum_{j,k=1}^\infty c_j\overline{c_k} u_j(x)\overline{u_k(y)} | v_j \rangle \langle v_k|.
	\end{equation}
	The identification is then extended by linearity to mixed states. Since $\Gamma(x,y)$ is of the form $|V(y)\rangle \langle U(x)|$, it is a trace-class operator on $\mathfrak{F}$ for almost every $(x,y)\in\mathbb{R}^d\times \mathbb{R}^d$. We also note that $\Gamma(x,x)\ge0$ for almost every $x\in\mathbb{R}^d$.\bigskip
	
	\noindent\textbf{Reduced density matrices.}	For generic states $\Gamma$ on $\mathfrak{H}$ we will define objects that monitor the state of the two subsystems (i.e., the particle and the bosonic field). Let us first recall the standard definitions of reduced density matrices in Fock space. Let $\Gamma_{\mathfrak{F}}$ be a state on $\mathfrak{F}$ satisfying
	\begin{equation*}
		\mathrm{Tr}_{\mathfrak{F}}\left[\mathcal{N}^{k}\Gamma_{\mathfrak{F}}\right]<+\infty.
	\end{equation*}
	for some $k\in\mathbb{N}$. For $p,q\in\mathbb{N}\cup\{0\}$, the $(p,q)$-reduced density matrix associated to $\Gamma_{\mathfrak{F}}$ is the operator 
	\begin{equation*}
	\Gamma_{\mathfrak{F}}^{(p,q)}:\left(L^2_\mathrm{field}(\mathbb{R}^d)\right)^{\otimes_\mathrm{sym}q} \to \left(L^2_\mathrm{field}(\mathbb{R}^d)\right)^{\otimes_\mathrm{sym}p}
	\end{equation*}
	defined by the relation
	\begin{equation} \label{eq:reduced_density_fock}
	\left\langle g_1\otimes_\mathrm{sym} \dots \otimes_\mathrm{sym} g_p,\Gamma_{\mathfrak{F}}^{(p,q)} f_1\otimes_\mathrm{sym}\dots \otimes_\mathrm{sym} f_q\right\rangle=\mathrm{Tr}_{\mathfrak{F}} \left( \Gamma_{\mathfrak{F}} \,a^\dagger(f_1)\dots a^\dagger (f_p) a(g_1)\dots a(g_q) \right)
	\end{equation}
	for $g_1,\dots,g_p,f_1,\dots,f_q \in L^2_\mathrm{field}(\mathbb{R}^d)$.
	We will mostly use $\Gamma_{\mathfrak{F}}^{(1,1)}$, $\Gamma_{\mathfrak{F}}^{(1,0)}$, and $\Gamma_{\mathfrak{F}}^{(0,1)}$. For the latter two the above definition reduces to 
	\begin{equation}
	\begin{split}
	\big\langle g,\Gamma_{\mathfrak{F}}^{(1,0)}\big\rangle=\;&\mathrm{Tr}_{\mathfrak{F}} \big[\Gamma_{\mathfrak{F}}\, a(g)\big]\\
	\big\langle \big(\Gamma_{\mathfrak{F}}^{(0,1)}\big)^*, f\big\rangle=\;&\mathrm{Tr}_{\mathfrak{F}}\big[  \Gamma_{\mathfrak{F}}\, a^\dagger (f)\big]= \overline{\langle f,[\gamma]^{(1,0)}\rangle}.
	\end{split}
	\end{equation}
	or, in terms of operator-valued distributions,
	\begin{equation}
	\Gamma_{\mathfrak{F}}^{(1,0)}(x)= \mathrm{Tr}_{\mathfrak{F}}[\gamma a_x]=\overline{[\gamma]^{(0,1)}(x)}.
	\end{equation}
	Moreover, for $\Gamma_{\mathfrak{F}}^{(1,1)}$ we have the relation
	\begin{equation}
	\begin{split}
	\mathrm{Tr}_{L^2_\mathrm{field}(\mathbb{R}^d)} \left( \Gamma_{\mathfrak{F}}^{(1,1)}\right) = \mathrm{Tr}_{\mathfrak{F}} \big( \mathcal{N}\Gamma_{\mathfrak{F}}\big).
	\end{split}
	\end{equation}
	
	We next define reduced density matrices for states on the full Hilbert space $\mathfrak{H}$.
	
	\begin{definition}[\textbf{Reduced density matrices for particle and field}]\mbox{} \label{def:reduced_densities} \\
	Let $\Gamma$ be a positive trace-class operator with unit trace on $\mathfrak{H}$, with the further property
	\begin{equation*}
		\mathrm{Tr}_{\mathfrak{H}} \big[\mathbbm{1}\otimes \mathcal{N}\,\Gamma\big]<+\infty.
	\end{equation*}
	We define the associated
	\begin{itemize}
		\item \emph{particle reduced density matrix} as the unit-trace, positive, trace-class operator $\gamma$ on $L^2_\mathrm{part}(\mathbb{R}^d)$ defined through the partial trace
		\begin{equation}
		\gamma=\mathrm{Tr}_{\mathfrak{F}}\big(\Gamma\big),
		\end{equation}
		or, equivalently, as the operator with integral kernel
		\begin{equation}			\gamma(x,y)=\;\mathrm{Tr}_{\mathfrak{F}}\big(\Gamma(x,y)\big).
		\end{equation}
		Notice that, as a consequence of \eqref{eq:Gamma_x_y} and its extension to mixed states, $\Gamma(x,y)$ is indeed trace-class for almost every $x,y\in\mathbb{R}^d$.
		\item \emph{field one-body reduced density matrix} as the unit-trace, positive, trace-class operator $\Gamma^{(1,1)}$ on $L^2_\mathrm{field}(\mathbb{R}^d)$ defined by
		\begin{equation} \label{eq:def_rho_field}
		\Gamma^{(1,1)}= \left[\mathrm{Tr}_{L^2_\mathrm{part}(\mathbb{R}^d)}\big(\Gamma\big)\right]^{(1,1)}=\left[\int_{\mathbb{R}^d}\Gamma(x,x) dx\right]^{(1,1)}.
		\end{equation}
		\item \emph{field-particle reduced density matrix} as the operator-valued linear map 
		\begin{equation*}\sigma: L^2_\mathrm{field}(\mathbb{R}^d)\to \cB (L^2_\mathrm{part}(\mathbb{R}^d))\end{equation*}
		whose action on a generic $f\in L^2_\mathrm{field}(\mathbb{R}^d)$ is defined by
		\begin{equation}
		\left\langle u, \sigma(f) v\right\rangle=\mathrm{Tr}\big( |v\rangle \langle u| \otimes \left( a^\dagger(f) + a (f) \right)\,\Gamma\big),\qquad\forall u,v\in L^2_\mathrm{part}(\mathbb{R}^d).
		\end{equation} 
		\end{itemize}
		\hfill$\diamond$
	\end{definition}
	
	We have the following properties of the field-particle reduced density matrix:
	
	\begin{lemma}[\textbf{Field-particle reduced density matrix}]\label{lem:redmat}\mbox{}\\
	 The field-particle reduced density matrix defined above takes values in the trace-class:
	 \begin{equation*} \sigma: L^2_\mathrm{field}(\mathbb{R}^d)\to \cL^1 (L^2_\mathrm{part}(\mathbb{R}^d)).\end{equation*}
	 Denote $(\sigma(f))(x,y)$ the integral kernel of $\sigma(f)$ and 
	 \begin{equation*}
		\sigma(x,x;z)=\mathrm{Tr}_{\mathfrak{F}}\big[\Gamma(x,x) \left( a^\dagger_z + a_z \right)\big]
	\end{equation*}
	as an operator-valued distribution satisfying
	\begin{equation*}
		(\sigma(f))(x,x)=\int_{\mathbb{R}^d}{{f(z)}}  \sigma(x,x;z)\,dz.
	\end{equation*}
	The distribution $\sigma(x,x;z)$ is in fact a function in $L^1_x\big(L^2_z(\mathbb{R^d})\big)$.
	\end{lemma}

	\begin{proof}
	To see that $\sigma(f)$ is indeed trace-class we may define it as an operator via the requirement
	\begin{align*} 
	\left\langle u,  \sigma (f) v \right\rangle &= \left\langle u, \sigma_+ (f) v \right\rangle - \left\langle u, \sigma_- (f) v \right\rangle\\
	&:= \mathrm{Tr}\big( |v \rangle \langle u| \otimes \left( a^\dagger(f) + a (f) \right)_+ \,\Gamma\big) - \mathrm{Tr}\big( |v \rangle \langle u| \otimes \left( a^\dagger(f) + a (f) \right)_- \,\Gamma\big)
	\end{align*}
	with $A_\pm$ the positive and negative parts of a self-adjoint operator. This way, if 
	\begin{equation*} \Tr \left( \1 \otimes \sqrt{\cN} \Gamma \right) < \infty\end{equation*}
	then $ \sigma (f)$ is the difference of two positive trace-class operators.

	Thus for every $f\in L^2_\mathrm{field}(\mathbb{R}^d)$, the integral kernel $(\sigma(f))(x,y)$ of $\sigma(f)$ is the function
	\begin{equation*}
		\left(\sigma(f)\right)(x,y)=\mathrm{Tr}_{\mathfrak{F}}\big[ \Gamma(x,y)\left( a^\dagger(f) + a (f) \right) \big].
	\end{equation*}
	Notice in particular that
	\begin{equation*}
		\int_{\mathbb{R}^d} |(\sigma(f))(x,x)|\,dx <+\infty
	\end{equation*}
	as is the case for the kernel of a trace-class operator. That $\sigma(x,x;z)$ is in $L^1_x\big(L^2_z(\mathbb{R^d})\big)$
	follows from
	\begin{equation} \label{eq:summability_sigma}
		\begin{split}
			\int_{\mathbb{R}^d}\left( \int_{\mathbb{R}^d}|\sigma(x,x;z)|^2dz\right)^{1/2}dx\le\;& C \int_{\mathbb{R}^d}\left( \int_{\mathbb{R}^d} \mathrm{Tr}_{\mathfrak{F}}\big(\Gamma(x,x)\big)\,\mathrm{Tr}_{\mathfrak{F}} \big( a^\dagger_z a_z \Gamma(x,x)\big)dz\right)^{1/2}dx\\
			\le\;& C \mathrm{Tr}_{\mathfrak{F}}\big(\mathbbm{1}\otimes(\mathcal{N}+1) \Gamma\big).
		\end{split}
	\end{equation}
	Here we used first the Cauchy-Scwharz inequality in the form  
	\begin{align*}\mathrm{Tr}_{\mathfrak{F}}\big[\Gamma(x,x) \left( a^\dagger_z + a_z \right)\big] &= \mathrm{Tr}_{\mathfrak{F}}\big[\Gamma(x,x)^{1/2} \Gamma(x,x)^{1/2} \left( a^\dagger_z + a_z \right)\big]\\
	 &\leq \mathrm{Tr}_{\mathfrak{F}}\big[\Gamma(x,x)\big] ^{1/2} \mathrm{Tr}_{\mathfrak{F}}\big[\Gamma(x,x)^{1/2}  \left( a^\dagger_z + a_z \right)^2 \Gamma(x,x)^{1/2}\big] ^{1/2}
	\end{align*}
recalling that $\left( a^\dagger_z + a_z \right)^2 \leq C a_z^\dagger a_z$. Next we used Cauchy-Schwarz for functions of $x$, together with the facts that 
\begin{align*}
 \int_{\R^d} \mathrm{Tr}_{\mathfrak{F}}\big[\Gamma(x,x) \big]dx &= \mathrm{Tr}_{\mathfrak{F}}\big[\Gamma \big]\\
 \int_{\R^d} \mathrm{Tr}_{\mathfrak{F}}\big[a_z^\dagger a_z \Gamma(x,x) \big]dz dx &= \mathrm{Tr}_{\mathfrak{F}}\big[\mathcal{N} \Gamma \big]   
\end{align*}
	
	\end{proof}

\subsection{Localization of states}	
	
	An important ingredient in the proof of our main result is the possibility to localize a generic state $\Gamma$ on $\mathfrak{H}$ to a certain region of $\mathbb{R}^d$ both for the particle's and for the quantized field's degrees of freedom. We here adapt to our coupled system the known construction for a single bosonic (or fermionic, for that matter) field~\cite{Ammari-04,DerGer-99,HaiLewSol-09a,HaiLewSol-09b,Lewin-11}.
	
	\begin{proposition}[\textbf{Construction of localized states}] \label{prop:def_localization}\mbox{} \\
		Let $0\le q \le \1$ be a linear operator on $L^2(\mathbb{R}^d)$, and $\Gamma$ be a positive trace-class operator on $\mathfrak{H}$. Let $\gamma,\Gamma^{(1,1)},\sigma$ be the reduced density matrices associated to $\Gamma$ according to Definition \ref{def:reduced_densities}.
		
		There exists a positive trace-class operator $\Gamma_{q}$ on $\mathfrak{H}$ whose reduced density matrices are
		\begin{align} \label{eq:localized_density_particle}
		\gamma_q=\;&q\,\gamma\,q\\
		\Gamma_q^{(1,1)}=\;&q\left[\mathrm{Tr}_{L^2_{\mathrm{part}}(\mathbb{R}^d)} \left( q \,\Gamma\, q\right) \right]^{(1,1)}q \label{eq:localized_density_field}
		\end{align}
		and, for $f\in L^2_\mathrm{field}(\mathbb{R}^d)$,
		\begin{align} \label{eq:localized_density_interaction}
		\sigma_{q}(f)=\;& q \,\sigma_{}(q f)\, q.
		\end{align}
		Moreover,
		\begin{equation} \label{eq:localized_traces}
			\mathrm{Tr}_{\mathfrak{H}} \Gamma = \mathrm{Tr}_{\mathfrak{H}} \Gamma_q+\mathrm{Tr}_{\mathfrak{H}}\Gamma_{(1-q^2)^{1/2}} .
		\end{equation}

	\end{proposition}
	
	\begin{proof}
		We follow and adapt the proof of \cite[Section~A.1.2]{HaiLewSol-09b}. Define the partial isometry
		\begin{equation}
		Q: L^2_\mathrm{field}(\mathbb{R}^d) \in f \mapsto q f \oplus(1-q^2)^{1/2} f \in L^2_\mathrm{field}(\mathbb{R}^d) \oplus L^2_\mathrm{field}(\mathbb{R}^d).
		\end{equation}
		and its' lifting to Fock space \footnote{The reader should think of the more familiar notation $\Gamma(Q)$ whose differential $\mathrm{d}\Gamma (Q)$ is the second quantization of $Q$. In our setting such a notation would clash with the way we are denoting states $\Gamma$ on $\mathfrak{H}$.}
		\begin{equation}
			\begin{split}
				&G(Q):\mathfrak{F}(L^2_\mathrm{field}(\mathbb{R}^d)) \to \mathfrak{F}\big(L^2_\mathrm{field}(\mathbb{R}^d)\oplus L^2_\mathrm{field}(\mathbb{R}^d)\big)\\
				&(G(Q)\Psi)^{(n)}=Q^{\otimes n} \,\Psi^{(n)}.
			\end{split}
		\end{equation}
		The latter operator satisfies
		\begin{equation}
			G(Q) \,a^\dagger(f)=a^\dagger\left( qf\oplus (1-q^2)^{1/2}f\right) G(Q).
		\end{equation}
		Recall now that there exists a canonical isomorphism
		\begin{equation}
			U:\mathfrak{F} \big(L^2_\mathrm{field}(\mathbb{R}^d)\oplus L^2_\mathrm{field}(\mathbb{R}^d)\big) \to \mathfrak{F}\big(L^2_\mathrm{field}(\mathbb{R}^d)\big) \otimes \mathfrak{F}\big(L^2_\mathrm{field}(\mathbb{R}^d)\big),
		\end{equation}
		and define the creation and annihilator operators on $
                \mathfrak{F}\big(L^2_\mathrm{field}(\mathbb{R}^d)\big) \otimes
                \mathfrak{F}\big(L^2_\mathrm{field}(\mathbb{R}^d)\big)$ as
		\begin{equation}
			\begin{split}
				&c^\dagger(f)= a^\dagger(f)\otimes \1_{\mathfrak{F}}\qquad c(f)=a(f)\otimes \1_{\mathfrak{F}}\\
				&d^\dagger(f)= \1_{\mathfrak{F}}\otimes a^\dagger(f)\qquad d(f)=\1_{\mathfrak{F}}\otimes a(f).
			\end{split}
		\end{equation}
		We denote with the same symbols the extensions of these operators to $\mathfrak{H}=L^2_\mathrm{part}(\mathbb{R}^d)\otimes \mathfrak{F}$ which act as the identity on $L^2_\mathrm{part}(\mathbb{R}^d)$.
		The relation between $U$, the latter creation and annihilation operators, and those on $\mathfrak{F}\big(L^2_\mathrm{field}(\mathbb{R}^d)\bigoplus L^2_\mathrm{field}(\mathbb{R}^d)\big)$ is
		\begin{equation}
			\begin{split}
				Ua^\dagger(f\oplus g)=\;&\big(c^\dagger(f)+d^\dagger(g)\big)U\\
				Ua(f\oplus g)=\;&\big(c(f)+d(g)\big)U.
			\end{split}
		\end{equation}
		Finally, define the operator
		\begin{equation}
			\mathcal{Y}(Q)= \1_{L^2_\mathrm{part}(\mathbb{R}^d)}\otimes \big(UG(Q)\big):\mathfrak{H} \to L^2_\mathrm{part}(\mathbb{R}^d) \otimes \mathfrak{F}\big(L^2_\mathrm{field}(\mathbb{R}^d)\big)  \otimes \mathfrak{F}\big(L^2_\mathrm{field}(\mathbb{R}^d)\big).
		\end{equation}
		By the relations above, $\mathcal{Y}(Q)$ satisfies the intertwining properties (see \cite[Lemma 2.14 and 2.15]{DerGer-99})
		\begin{equation} \label{eq:intertwining}
			\begin{split}
			\mathcal{Y}(Q)\,a^\dagger(f)=\;&\Big( c^\dagger(qf)+d^\dagger\big((1-q^2)^{1/2}f\big) \Big)	\mathcal{Y}(Q)\\
			\mathcal{Y}(Q)\,a(f)=\;&\Big( c(qf)+d\big((1-q^2)^{1/2}f\big) \Big)	\mathcal{Y}(Q)\\
			\mathcal{Y}(Q)a(qf)=c(f)\mathcal{Y}&(Q)\qquad \mathcal{Y}(Q)a\big((1-q^2)^{1/2}f\big)=d(f)\mathcal{Y}(Q)\\
			a^\dagger(qf)\mathcal{Y}(Q)^*= \mathcal{Y}(Q)^*c^\dagger&(f)\qquad a^\dagger\big((1-q^2)^{1/2}f\big)\mathcal{Y}(Q)^*= \mathcal{Y}(Q)^*d^\dagger(f)
			\end{split}
		\end{equation}
		Moreover, since $Q^*Q=\1_{L^2_\mathrm{field}(\mathbb{R}^d)}$, it follows that $\mathcal{Y}(Q)^*\mathcal{Y}(Q)=\1_{\mathfrak{H}}$.
%
%
%
%
		
		Let now $\Gamma$ be a positive trace class operator on $\mathfrak{H}$. We define the $q$-localization of $\Gamma$ as the trace-class operator $\Gamma_q$ on $\mathfrak{H}$ whose action on a factorized bounded operator $A\otimes B \in \mathcal{B}(L^2_\mathrm{part}(\mathbb{R}^d))\otimes \mathcal{B}(\mathfrak{F})$ is
		\begin{equation}
			\mathrm{Tr}_{\mathfrak{H}}\left[ (A\otimes B) \Gamma_q \right]=\mathrm{Tr}_{\mathfrak{H}}\Big[\mathcal{Y}(Q)^* (qAq \otimes B \otimes \1_{\mathfrak{F}}) \mathcal{Y}(Q)\Gamma\Big],
		\end{equation}
		and the extension to non-factorized operators follows by linearity.
		The fact that $\Gamma_q$ is positive follows from the positivity of $\Gamma$. Moreover, the identity $\mathcal{Y}(Q)^* \mathcal{Y}(Q)=\1_{\mathfrak{H}}$ implies that
		\begin{equation}
			\mathrm{Tr}_{\mathfrak{H}} \Gamma_q=\mathrm{Tr}_{\mathfrak{H}} \left[\big(q^2\otimes \1_{\mathfrak{F}}\big)\Gamma\right].
		\end{equation}
		Repeating the construction by switching the roles of $q$ and $(1-q^2)^{1/2}$ we similarly find
		\begin{equation}
			\mathrm{Tr}_{\mathfrak{H}} \Gamma_{(1-q^2)^{1/2}}= \mathrm{Tr}_{\mathfrak{H}}\left[ \big((1-q^2)\otimes \1_{\mathfrak{F}}\big)\Gamma\right].
		\end{equation}
		The last two identities prove \eqref{eq:localized_traces}.
		
%
		
		In order to show \eqref{eq:localized_density_particle} we compute
		\begin{equation}
		\big\langle g,\gamma_q f \big\rangle =\mathrm{Tr}_{\mathfrak{H}}\big( \ket{f}\bra{g} \otimes \1_{\mathfrak{F}} \Gamma_q\big )=\mathrm{Tr}_{\mathfrak{H}} \big(  \ket{qf} \bra{qg} \otimes \1_{\mathfrak{F}}\Gamma \big)= \big\langle f,q\, \gamma\, q\, g \big\rangle.
		\end{equation}
		This is precisely \eqref{eq:localized_density_particle}. In order to show \eqref{eq:localized_density_field}, in turn, we recall \eqref{eq:def_rho_field} and \eqref{eq:reduced_density_fock} to write
		\begin{equation}
			\big\langle g, \Gamma^{(1,1)}_qf\big\rangle =  \Big\langle g,  \left[\mathrm{Tr}_{L^2_\mathrm{part}(\mathbb{R}^d)}(\Gamma_q)\right]^{(1,1)}f \Big\rangle = \mathrm{Tr}_{\mathfrak{H}} \left[\left(\1_{L^2_\mathrm{part}(\mathbb{R}^d)}\otimes a^\dagger(f) a(g)\right) \Gamma_q \right].
		\end{equation}
		Eq. \eqref{eq:localized_density_field} is then deduced using the definition of $\Gamma_q$ and the intertwining properties \eqref{eq:intertwining}. Finally, in order to show \eqref{eq:localized_density_interaction} we write, for a generic $O\in \mathcal{B}\big(L^2_\mathrm{part}(\mathbb{R}^d)\big)$,
		\begin{equation*}
			\mathrm{Tr}_{L^2_\mathrm{part}(\mathbb{R}^d)}\big[ O\, \sigma_{q}(f) \big]=\mathrm{Tr}_{\mathfrak{H}}\left[\big(O \otimes a(f)\big)\Gamma_q\right]
		\end{equation*}
		Again, the definition of $\Gamma_q$ and the intertwining properties allow to conclude.
	\end{proof}

	\subsection{Energy localization}\label{sec:loc}

	We fix a smooth partition of unity $\chi^2+\eta^2=1$ with $\chi(x)=1$ if $|x|\le 1$ and $\chi(x)=0$ if $|x|\ge2$, and define $\chi_R(x)=\chi(x/R)$ and $\eta_R(x)=\eta(x/R)$. We further assume that $\chi$, and thus $\eta$, are monotone functions. We then have

	\begin{proposition}[\textbf{Energy localization}] \mbox{}\label{prop:localization} \\ 
		Let $V,v$ be as in Assumptions \ref{assum:V} and \ref{assum:int}, and let $H_\alpha^{(V)}$ be defined in \eqref{eq:H_alpha^V}. Consider a family $(\Psi_\alpha)_{\alpha}$ of normalized vectors $\Psi_{\alpha}\in \mathfrak{H}_{\alpha}$ and the associated states $\Gamma_\alpha=\ket{\Psi_\alpha}\bra{\Psi_\alpha}$. Assume that 
		\begin{equation*}\mathrm{Tr}\big[H_\alpha^{(V)}\, \Gamma_\alpha\big]\le C\end{equation*}
		uniformly as $\alpha\to\infty$. Let $\gamma_\alpha,\Gamma_\alpha^{(1,1)},\sigma_\alpha$ be the reduced density matrices associated to $\Gamma_\alpha$ according to Definition \ref{def:reduced_densities}. Let $\Gamma_{\alpha,\chi_R},\Gamma_{\alpha,\eta_R}$ be the localized states corresponding to the choices $\Gamma=\Gamma_\alpha$ and $q=\chi_R,\eta_R$ in Proposition \ref{prop:def_localization}. Then
		\begin{equation} \label{eq:full_energy_localization}
			\begin{split}
				\liminf_{\alpha\to\infty}&\, \mathrm{Tr}_{\mathfrak{H}}\left(H_\alpha^{(V)}\, \Gamma_\alpha\right) \\
				\ge\;& \liminf_{R\to\infty} \liminf_{\alpha\to\infty} \Big[ \mathrm{Tr}_{\mathfrak{H}}\left(H_\alpha^{(V)}\, \Gamma_{\alpha,\chi_R}\right) + \mathrm{Tr}_{\mathfrak{H}}\left(H_\alpha^{(0)}\, \Gamma_{\alpha,\eta_R}\right)\Big].
			\end{split}
		\end{equation}
	\end{proposition}

We split the proof into three lemmas, corresponding to the three terms in the energy.
%

\begin{lemma}[\textbf{Particle energy localization}] \mbox{}\label{lemma:part_localization} \\
	In the same assumptions of Proposition \ref{prop:localization} we have
	\begin{equation} \label{eq:splitting_particle}
		\begin{split}
			\liminf_{\alpha\to\infty} \mathrm{Tr}_{\mathfrak{H}}\Big((-\Delta+V)\otimes \mathbbm{1}\,\Gamma_\alpha\Big)\ge\;& \liminf_{R \to\infty} \liminf_{\alpha\to\infty}\Big[ \mathrm{Tr}_{L^2_\mathrm{part}(\mathbb{R}^d)}\big(\chi_R (-\Delta+V )\chi_R\,\gamma_\alpha\big)\\
			&\qquad\qquad\qquad\quad+ \mathrm{Tr}_{L^2_\mathrm{part}(\mathbb{R}^d)}\big(-\eta_R\Delta\eta_R \gamma_\alpha\big) \Big].
		\end{split}
	\end{equation}
\end{lemma}

\begin{proof}
	Let us first focus on proving a lower bound on the term involving $-\Delta$. Using the IMS formula
	\begin{equation*}
		-\Delta=-\chi_R \Delta \chi_R-\eta_R \Delta \eta_R-|\nabla \chi_R|^2 - |\nabla \eta_R|^2.
	\end{equation*}
	The two gradient terms are bounded functions which, by the definition of $\chi_R$ and $\eta_R$, satisfy
	\begin{equation*}
		|\nabla\chi_R|^2+|\nabla \eta_R|^2 \le \frac{C}{R^2}.
	\end{equation*}
	This immediately implies, using also the definition of $\gamma_\alpha$,
	\begin{equation*}
		\begin{split}
			\mathrm{Tr}_{\mathfrak{H}}\big( -\Delta\otimes\mathbbm{1}\,\Gamma_\alpha \big)=\;&\mathrm{Tr}_{L^2_\mathrm{part}(\mathbb{R}^d)}\big( -\Delta\, \gamma_\alpha\big)\\
			\ge\;&\mathrm{Tr}_{L^2_\mathrm{part}(\mathbb{R}^d)}\big( -\chi_R\Delta\chi_R \,\gamma_\alpha\big)
			+ \mathrm{Tr}_{L^2_\mathrm{part}(\mathbb{R}^d)}\big( -\eta_R\Delta \eta_R \,\gamma_\alpha\big )-\frac{C}{R^2}.
		\end{split}
	\end{equation*}
	Passing to the $\liminf$ for $\alpha\to\infty$ followed by $R\to\infty$, we conclude
	\begin{equation*}
		\begin{split}
			\liminf_{\alpha\to\infty} \mathrm{Tr}_{\mathfrak{H}}\big( -\Delta\otimes\mathbbm{1}\,\Gamma_\alpha \big)\ge\;& \liminf_{R \to\infty} \liminf_{\alpha\to\infty}\Big[ \mathrm{Tr}_{L^2_\mathrm{part}(\mathbb{R}^d)}\big(-\chi_R \Delta\chi_R\,\gamma_\alpha\big)+ \mathrm{Tr}_{L^2_\mathrm{part}(\mathbb{R}^d)}\big(-\eta_R\Delta\eta_R\gamma_\alpha\big) \Big].
		\end{split}
	\end{equation*}
	For the $V$-term in \eqref{eq:splitting_particle} we proceed in a similar way, by writing
	\begin{equation*}
		V=\chi_R^2 V + \eta_R^2 V.
	\end{equation*}
	Since $V$ is a bounded function that decays at infinity, we have, as $R\to\infty$, $V \eta_R^2 \ge - \,o_R(1)$. Passing to the two $\liminf$'s concludes the proof.
\end{proof}

We next localize the field energy. We could generalize this to more general field dispersion relations using appropriate IMS formulas, cf the discussion following~\eqref{eq:H_alpha^V}.

\begin{lemma}[\textbf{Field energy localization}] \mbox{}\label{lemma:field_localization} \\
	In the same assumptions of Proposition \ref{prop:localization} we have
	\begin{equation} \label{eq:splitting_field}
		\begin{split}
			\liminf_{\alpha\to\infty}& \mathrm{Tr}_{\mathfrak{H}}\big( \mathbbm{1}\otimes \mathcal{N}\,\Gamma_\alpha\big)\\
			\ge\;& \liminf_{R \to\infty} \liminf_{\alpha\to\infty}\bigg\{\mathrm{Tr}_{L^2_\mathrm{field}(\mathbb{R}^d)}\bigg[ \chi_R\bigg(\int_{\mathbb{R}^d} \chi_R^2(x) \Gamma_\alpha(x,x) dx\bigg)^{(1,1)}\chi_R\bigg]\\
			&\qquad\qquad\qquad+\mathrm{Tr}_{L^2_\mathrm{field}(\mathbb{R}^d)}\bigg[ \eta_R\bigg(\int_{\mathbb{R}^d} \eta_R^2(x) \Gamma_\alpha(x,x) dx\bigg)^{(1,1)}\eta_R\bigg]\,\bigg\}.
		\end{split}
	\end{equation}
\end{lemma}

\begin{proof}
	Recalling the identification~\eqref{eq:Gamma_x_y} together with~\eqref{eq:def_rho_field} we have
	\begin{align*}
		\mathrm{Tr}_{\mathfrak{H}}\Big( \mathbbm{1}\otimes \mathcal{N}\,\Gamma_\alpha\Big)&= \int \mathrm{Tr}_{\mathfrak{F}} \left( \mathcal{N} \Gamma_\alpha (x,x)\right) dx\\
		&= \mathrm{Tr}_{L^2_\mathrm{field}(\mathbb{R}^d)}\left(\left[ \int_{\mathbb{R}^d} \Gamma_\alpha(x,x)dx\right]^{(1,1)}\right)
	\end{align*}
	by definition of the first reduced density matrix of a state on Fock space. Using the fact that 
	$$\chi_R^2+\eta_R^2=1$$
	twice and discarding the two mixed terms for a lower bound (recall that $\Gamma(x,x)\ge0$) leads to
	\begin{equation*}
		\begin{split}
			\mathrm{Tr}_{\mathfrak{H}}\big( \mathbbm{1}\otimes \mathcal{N}\,\Gamma_\alpha\Big) \ge \;&\mathrm{Tr}_{L^2_\mathrm{field}(\mathbb{R}^d)}\left(\chi_R \left[\int_{\mathbb{R}^d}\chi_R^2(x) \Gamma_\alpha(x,x)dx\right]^{(1,1)}\chi_R\right)\\
			&+\mathrm{Tr}_{L^2_\mathrm{field}(\mathbb{R}^d)}\left(\eta_R \left[\int_{\mathbb{R}^d}\eta_R^2(x) \Gamma_\alpha(x,x)dx\right]^{(1,1)}\eta_R\right).
		\end{split}
	\end{equation*}
\end{proof}

Finally we deal with the particle-field interaction:

\begin{lemma}[\textbf{Interaction energy localization}] \mbox{}\label{lemma:interaction_localization} \\
	Under the same assumptions as in Proposition \ref{prop:localization} we have
	\begin{equation} \label{eq:splitting_interaction}
		\begin{split}
			\liminf_{\alpha\to\infty}&\; \mathrm{Tr}_{\mathfrak{H}}\bigg( \int_{\mathbb{R}^d}v(\cdot-z) \left( a^\dagger_z + a_z\right) \,\Gamma_\alpha\bigg)\\
			\ge\;& \liminf_{R \to\infty} \liminf_{\alpha\to\infty}\bigg[\iint_{\mathbb{R}^{2d}} \chi_R^2(x)\chi_R(z) v(x-z) \sigma_\alpha(x,x;z)\, dxdz\\
			&\qquad\qquad\qquad\quad +\iint_{\mathbb{R}^{2d}} \eta_R^2(x)\eta_R(z) v(x-z) \sigma_\alpha(x,x;z)\, dxdz \bigg]
		\end{split}
	\end{equation}
\end{lemma}

	\begin{proof}
		First, by definition of $\sigma_\alpha$
		\begin{equation*}
			\mathrm{Tr}_{\mathfrak{H}}\bigg( \int_{\mathbb{R}^d}v(\cdot-z) \left( a^\dagger_z + a_z\right) \,\Gamma_\alpha\bigg)=\iint_{\mathbb{R}^{2d}}v(x-z)  \sigma_\alpha(x,x;z)\, dxdz.
		\end{equation*}
		Using $\chi_R^2+\eta_R^2=1$ 
		we have
		\begin{equation}
			\begin{split}
				\iint_{\mathbb{R}^{2d}}v(x-z)\sigma_{\alpha}(x,x;z)dxdz =\;& \iint_{\mathbb{R}^{2d}} \chi_R^2(x) \chi_R(z) v(x-z)\sigma_{\alpha}(x,x;z)\,dxdz\\
				&+\iint_{\mathbb{R}^{2d}} \eta_R^2(x)\eta_R(z) v(x-z) \sigma_\alpha(x,x;z)\, dxdz\\
				&+\mathcal{E}_\mathrm{Int}^{(1)}+\mathcal{E}_\mathrm{Int}^{(2)},
			\end{split}
		\end{equation}
		with
		\begin{equation}
			\begin{split}
				\mathcal{E}_\mathrm{Int}^{(1)} =\;&\iint_{\mathbb{R}^{2d}} \chi_R^2(x) (1-\chi_R(z))v(x-z)\sigma_{\alpha}(x,x;z)\,dxdz\\
				\mathcal{E}_\mathrm{Int}^{(2)} =\;& \iint_{\mathbb{R}^{2d}}\eta_R^2(x) (1-\eta_R(z))v(x-z)\sigma_{\alpha}(x,x;z)\,dxdz.
			\end{split}
		\end{equation}
		Let us show that these are negligible in the limit $\alpha\to\infty$ followed by $R\to\infty$. The two terms are treated similarly, starting with $\mathcal{E}_\mathrm{Int}^{(1)}$. First, we have 
		\begin{equation*}1-\chi_R= \frac{\eta_R^2}{1+\chi_R} \le \eta_R^2.\end{equation*} 
		In addition, since 
		\begin{equation*}\eta_R^2=\eta_{4R}^2+ \eta_R^2-\eta_{4R}^2,\end{equation*}
		we have the bound
		\begin{equation} \label{eq:partial_decomp_E^1_int}
			\begin{split}
				|\mathcal{E}_\mathrm{Int}^{(1)}|\le\;&\iint_{\mathbb{R}^{2d}} \chi_R^2(x) \eta_{4R}^2(z)|v(x-z)|\,|\sigma_{\alpha}(x,x;z)|dxdz\\
				&+	\iint_{\mathbb{R}^{2d}}  \chi_R^2(x)\left( \eta_R^2(z)- \eta_{4R}^2(z)\right)|v(x-z)|\,|\sigma_{\alpha}(x,x;z)|dxdz.		
			\end{split}
		\end{equation}
		To control the first term in the right hand side we notice that $\chi_R^2(x) \eta_{4R}^2(z) \le \1_{\{|x-z|\ge R\}}$, and therefore, by Cauchy-Schwarz,
		\begin{equation*}
			\begin{split}
			\iint_{\mathbb{R}^{2d}} \chi_R^2(x)& \eta_{4R}^2(z)|v(x-z)|\,|\sigma_{\alpha}(x,x;z)|dxdz \\
			\le\;& \int_{\mathbb{R}^d} \bigg(\int_{\{|z-x|\ge R\}}|v(x-z)|^2dz \bigg)^{1/2} \bigg( \int_{\mathbb{R}^d}|\sigma_\alpha(x,x;z)|^2dz \bigg)^{1/2}dx\\
			\le\;& o_R(1)
			\end{split}
		\end{equation*}
		uniformly in $\alpha$. Here we have used the fact that $v\in L^2(\mathbb{R}^d)$, as well as \eqref{eq:summability_sigma} and the fact that the energy of $\Gamma_\alpha$ is uniformly bounded by assumption.
		
		For the second term in the decomposition \eqref{eq:partial_decomp_E^1_int} of $\mathcal{E}^{(1)}_\mathrm{Int}$ we argue using an adaptation of Lions' concentration-compactness argument, already used in \cite[Lemma 4.8]{LewNamRou-14}. Let us define the function
		\begin{equation*}
			Q_\alpha(R)=\iint_{\mathbb{R}^{2d}} |v(x-z)|\,\1_{|z| \ge R}\, |\sigma_{\alpha}(x,x;z)|dxdz.
		\end{equation*}
		Then (recall that $\eta_R^2 \ge \eta_{4R}^2$ since $\eta$ is monotone)
		\begin{equation*}
			\iint_{\mathbb{R}^d} \chi_R^2(x)\left( \eta_R^2(z)- \eta_{4R}^2(z)\right)|v(x-z)|\,|\sigma_{\alpha}(x,x;z)|dxdz \le Q_\alpha(R)-Q_\alpha(8R).
		\end{equation*}
		Now, for fixed $\alpha$, the function $R\mapsto Q_\alpha(R)$ is non-increasing on $[0,\infty)$, and
		\begin{equation*}
			0\le Q_\alpha(R) \le \|v\|_{L^2} \int_{\mathbb{R}^d}\left( \int_{\mathbb{R}^d}|\sigma_\alpha(x,x;z)|^2dz \right)^{1/2}dx\le C_1
		\end{equation*}
		uniformly in $\alpha$ and $R$ thanks to \eqref{eq:summability_sigma} and the fact that the energy of $\Gamma_\alpha$ is uniformly bounded. We apply the above along a sequence in $\alpha$ attaining the limsup of $Q_\alpha (R) - Q_\alpha (8R)$.  Then, by Helly's selection principle, there exists a subsequence $\alpha_k$ and 	a decreasing function $Q:[0,\infty)\to [0,C_1]$ such that
		\begin{equation*}
			\lim_{k\to\infty} Q_{\alpha_k}(R)=Q(R),\qquad\forall R \in[0,\infty).
		\end{equation*}
		Since $\lim_{R\to\infty} Q(R)$ exists by monotonicity and is finite, we conclude
		\begin{equation*}
			\lim_{R \to\infty} \limsup_{\alpha \to \infty}\left( Q_{\alpha}(R)- Q_{\alpha}(8R) \right) = \lim_{R \to\infty} \lim_{k \to \infty}\left( Q_{\alpha_k}(R)- Q_{\alpha_k}(8R) \right) =\lim_{R\to\infty}\left( Q(R)-Q(8R) \right)=0.
		\end{equation*}
		This implies that 
		$$\lim_{R \to\infty} \lim_{\alpha \to \infty}\left( Q_{\alpha}(R)- Q_{\alpha}(8R) \right) = 0$$
		because the left-hand side is always non-negative, and we conclude that
		\begin{equation*}
			\lim_{R \to\infty}\lim_{\alpha\to\infty}\mathcal{E}^{(1)}_\mathrm{Int}=0
		\end{equation*}
		 We argue similarly to obtain $\mathcal{E}_\mathrm{Int}^{(2)} \to 0$.
	\end{proof}

	We now conclude the 
	\begin{proof}[Proof of Proposition \ref{prop:localization}]
		The result follows immediately from Lemma \ref{lemma:part_localization}, \ref{lemma:field_localization}, and \ref{lemma:interaction_localization} after recalling the expressions of the reduced density matrices of the localized states $\Gamma_{\alpha,\chi_R}$ and $\Gamma_{\alpha,\eta_R}$ from Proposition \ref{prop:def_localization}.
	\end{proof}
	
	\section{Quasi-classical measures} \label{sect:deF}

	We revisit the construction of quasi-classical measures from~\cite{Falconi-18a,Falconi-18b,CorFalOli-19,CorFalOli-20}, linking them with the approach of~\cite{LewNamRou-14d}. Slightly improved statements are obtained by using anti-Wick rather than Weyl quantization in the basic definition of the measures, but otherwise the spirit is extremely similar. Related statements and ideas may be found in~\cite{FanLewVer-88}.
	
	\subsection{Notation}

For a complex separable Hilbert space $\gH$ we denote $\cB (\gH)$ the set of bounded operators acting thereon, $\cB (\gH)^*$ its dual and $\cS (\gH)$ the state-space, i.e. 
\begin{equation}\label{eq:states}
 \cS (\gH) := \left\{ \omega \in \cB (\gH)^*, \omega (B) \geq 0 \mbox{ for all } 0 \leq B \in \cB (\gH), \omega (\1) = 1 \right\}.
\end{equation}
These are ``abstract states'' by opposition to trace-class operators, i.e. normal states. One advantage in considering them is that a sequence of abstract states always has a weak-$\star$ cluster point which is a state. A bit of care is needed in using the weak-$\star$ topology on $\cB (\gH)^*$ because the pre-dual $\cB (\gH)$, is not, in infinite dimension, separable. The compactness of sequences of states thus takes the form that (by the Banach-Alaoglu Theorem) given a sequence $(\omega_n)_n\in \cS (\gH)^\bN$ there is a $\omega\in \cS (\gH)$ such that $\omega_n$ converges to $\omega$ along a subnet. This means that for any $B\in \cB(\gH)$
\begin{equation}\label{eq:extract abstract}
\omega_{h(\alpha)} (B) \to \omega (B) 
\end{equation}
where $h:A\mapsto \mathbb{N}$ is a monotone cofinal function from some directed set $A$ to the integers. It is important to be able to test against the identity operator in \eqref{eq:extract abstract}, to ensure that $\omega$ is a state. With an abuse of notation we denote this convergence by 
\begin{equation}\label{eq:CV net}
\omega_n \cvnet \omega
\end{equation}
where an extraction is implied.

\subsection{The theorem}

Let $\gh,\gH$ be two separable complex Hilbert spaces. We are interested in states of the composite system with Hilbert space 
\begin{equation*} 
\gHt := \gh \otimes \gF(\gH)
\end{equation*}
where $\gF(\gH)$ is the bosonic Fock space constructed from $\gH$. We denote by $\cN$ the number operator on $\gF(\gH)$ and 
\begin{equation*} \gH_n := \gH^{\otimes_s n}\end{equation*}
the $n$-particles sector. For a state $\Gamma$ on $\gHt$, $ \langle \, O \, \rangle_{\Gamma} $ denotes the expectation value of $O$ in $\Gamma$.

For facilitated comparison with~\cite{LewNamRou-14d} we here follow the convention that annihilation and creation operators are unscaled (contrarily to the convention in~\eqref{eq:scaleop}), so that the CCR takes the form~\eqref{eq:CCR}
\begin{equation} \label{eq:CCRagain}
[ c(f),c(g)]=[c^\dagger(f),c^\dagger(g)]=0,\qquad[c(f),c^\dagger(g)]=\langle f,g\rangle, \qquad \forall f,g\in L_\mathrm{field}^2(\mathbb{R}^d)
\end{equation}
for the creation and annihilation operators (cf~\eqref{eq:annihi}) 
\begin{equation*}c^\dagger (f) = \alpha \ada (f) , \quad c(f)= \alpha a(f).\end{equation*}

\begin{definition}[\textbf{Reduced density matrices}]\label{def:sym state}\mbox{}\\
Let $\Gamma \in \cS (\gHt)$ be a state over $\gHt$. We define reduced densities $\Gamma ^{(k,\ell)}$ as maps from $\cB(\gh)$ to $\cB (\gH_\ell,\gH_k)$ (the bounded operators from $\gH_\ell$ to $\gH_k$) by the formula
\begin{equation}\label{eq:red dens}
\left\langle f_1\otimes_s \ldots \otimes_s f_k |\Gamma ^{(k,\ell)} (A)  g_1\otimes_s \ldots \otimes_s g_\ell\right\rangle := \left\langle A \otimes \cda (g_1) \ldots \cda (g_\ell) c (f_1) \ldots c(f_k) \right\rangle_\Gamma
\end{equation}
where $A\in \cB (\gh)$ and $f_1,\ldots,f_k,g_1,\ldots,g_\ell \in \gH$. The definition makes sense as soon as 
\begin{equation*} \left\langle \1\otimes \cN ^{\frac{k+\ell}{2}}\right\rangle_\Gamma < \infty\end{equation*}
where 
\begin{equation*} \cN = \sum_{j\geq 1} \cda (f_j) c(f_j)\end{equation*}
for any orthonormal basis $(f_j)_j$ of $\gH$.\hfill$\diamond$
\end{definition}

\begin{definition}[\textbf{Anti-Wick observables}]\label{def:a-Wick}\mbox{}\\
To any $u\in \gH$  we associate a \emph{coherent state} on $\gF(\gH)$
\begin{equation}
\xi(u):=e^{-\frac{|u|^2}{2}} \bigoplus_{j \ge 0} \frac{1}{\sqrt{j!}} u^{\otimes j}\in\gF(\gH).
\label{eq:def_coherent}
\end{equation}
For any sequence $\eps\to0$ and any finite-dimensional subspace $V$ of $\gH$ we define the \emph{anti-Wick quantization} of a continuous function with compact support $b\in C^0_c(V)$ at scale $\eps$, by
\begin{equation}
b^{\rm aW}_{\eps}:=(\eps\pi)^{-\dim(V)}\int_{V} b(u)\;\left|\xi\left(u/\sqrt{\eps}\right)\right\rangle\left\langle\xi\left(u/\sqrt{\eps}\right)\right|\, du.
\label{eq:def_quantization_b}
\end{equation}
\hfill$\diamond$
\end{definition}
%

We aim at proving the 

\begin{theorem}[\textbf{Quantum de Finetti for composite systems}]\label{thm:deF comp}\mbox{}\\
Consider a sequence $\eps \to 0$ of positive parameters, and associated sequence $\Gamma_\eps$ of states over $\gHt$ satisfying
\begin{equation} 
\left\langle \left(\eps \cN \right)^\kappa \right\rangle_{\Gamma_{\eps}} < + \infty \label{eq:control_N} \end{equation}
uniformly in $\eps$, for some $1 \leq \kappa$. 

There exists a probability measure $\mu \in \cP (\gH)$ and a $\mu$-measurable map 
\begin{equation*} \omega : \begin{cases}
             \gH \to \cS (\gh)\\
             u \mapsto \omega_u
            \end{cases}
\end{equation*}
with values in the state-space of $\gh$ such that, 
\begin{enumerate}
 \item Expectations of anti-Wick observables converge and define the measure:  along a subnet, for all $A\in \cB(\gh)$, $V \subset \gH$ a finite-dimensional subspace and all $b\in C^0_c (V)$ (continuous functions with compact support in $V$) we have 
 \begin{equation}\label{eq:AW CV comp}
  \left\langle A \otimes b^{\rm aW}_{\eps} \right\rangle_{\Gamma_\eps} \to \int_{ \gH} \omega_u (A) b(u) d\mu (u) 
 \end{equation}
\item Reduced density matrices converge: along a subsequence, for $A$ a compact operator or the identity\footnote{In fact, modulo a subsequence, we can test with $A$ in any separable subspace of the bounded operators.}
\begin{equation}\label{eq:CV DM comp}
\sqrt{k! \ell ! \eps ^k \eps^\ell}\Gamma_\eps ^{(k,\ell)} (A) \underset{\eps\to 0}{\wto} \int_{\gH}  \omega_u (A) |u^{\otimes k} \rangle \langle u ^{\otimes \ell}| d\mu (u)
\end{equation}
weakly-$\star$ in the trace-class for all $k,\ell$ satisfying $\frac{k+\ell}{2} \leq \kappa$. More precisely
\begin{equation*}
\sqrt{k! \ell ! \eps ^k \eps^\ell}\left\langle A \otimes \cda (g_1) \ldots \cda (g_\ell) c (f_1) \ldots c(f_k) \right\rangle_{\Gamma_{\eps}} \underset{\eps\to 0}{\to} \int_{\gH}  \omega_u (A) \prod_{j=1} ^k \langle f_j | u \rangle \prod_{j=1}^\ell \langle u | g_j \rangle d\mu (u)
\end{equation*}
for all $f_1,\ldots,f_k,g_1,\ldots,g_\ell \in \gH$.
\end{enumerate}
\end{theorem}

We shall rely on a version of the above in the non-composite case (where $\gh \otimes \gF(\gH)$ is replaced by $\gF(\gH)$) from~\cite[Sections 4 and 6]{LewNamRou-14d}. This is also contained in~\cite{AmmNie-08,AmmNie-09} where the  construction is rather based on Weyl observables/quantizations rather than anti-Wick as we use here. 

Before proceeding to the proof, we state as corollary the convergence of observables akin to the interaction energy of our main model. 

\begin{corollary}[\textbf{Quantum de Finetti and the particle-field density matrix}]\label{cor:deFint}\mbox{}\\
Let $\gh = L^2 (\R^d)$, $v\in L^2 (\R^d,\R)$. Under the assumptions above (for $\kappa = 1$), after extracting a subsequence
\begin{equation}\label{eq:CV mixed DM}
\sqrt{\eps} \left\langle \int_{\R^d} v(\cdot - y) \left( \cda_y + c_y \right) dy \right\rangle_{\Gamma_{\eps}} \underset{\eps\to 0}{\to} \int_{\gH} \int_{\R^d} \omega_u (v(\cdot-y)) \left( u (y) + \overline{u(y)}\right) dy \,d\mu (u)
\end{equation}
where $v(\cdot-y)$ is understood as a multiplication operator on $\gh$. In other words 
\begin{equation}\label{eq:CV mixed DM bis}
\iint_{\R^2 \times \R^2} \sqrt{\eps}\sigma_{\eps} (x,x;z) v(x-z) dx dz \underset{\eps\to 0}{\to} \int_{\gH} \int_{\R^d} \omega_u (v(\cdot-z)) \left( u (z) + \overline{u(z)}\right) dz\, d\mu (u)
\end{equation}
where $\sigma_{\eps} (x,y;z)$ is the integral kernel of the field-particle density matrix of $\Gamma_\eps$, as defined in Section~\ref{sec:DMs}.
\end{corollary}

\begin{proof}
By arguments mimicking~\eqref{eq:summability_sigma} we find that 
\begin{align*}
\int_{\R^d} \left( \int_{\R^d} \left|\sqrt{\eps}\sigma_{\eps} (x,x;z)\right|^2 dz \right) ^{1/2}dx &\leq C \int_{\R^d} \mathrm{Tr}_\gF \left[\Gamma_\eps (x,x)\right] ^{1 / 2} \left(\int_{\R^d} \mathrm{Tr}_\gF [\Gamma_\eps (x,x) \eps a^\dagger_z a_z] dz \right)^{1/2} dx \\
&\leq C\left\langle \1 \otimes \left(\eps \cN +1 \right)\right\rangle_{\Gamma_\eps}.
\end{align*}
Hence $(x,z)\mapsto \sqrt{\eps} \sigma_{\eps} (x,x;z)$ is uniformly bounded as a sequence in $L^1_x (L^2_z (\R^d))$, which is a subset of the dual of the Banach space $C_x^{0,b} (L^2 (\R^d))$ of bounded continuous functions with values in $L^2 (\R^d)$ (see e.g.~\cite{HytNeeVerWei-16}). Hence, modulo extraction of a subsequence 
\begin{equation*} 
\sqrt{\eps}\iint_{\R^2 \times \R^2} \sigma_{\eps} (x,x;z) \phi(x,z) dx dz \underset{\eps \to 0}{\to} \iint_{\R^2 \times \R^2} \sigma_{0} (x,x;z) \phi(x,z) dx dz
\end{equation*}
for any $\phi\in C_x^0 (L^2 (\R^d)),$ with $\sigma_0$ a Radon measure over $L^2 (\R^d)$ (with a slight abuse of notation in the right-hand side of the above). For $v\in L^2 (\R^d)$, the map $(x,z) \mapsto v(x-z)$ is  in $C_x^{0,b} (L^2_z (\R^d))$ since the statement 
\begin{equation*}
			\lim_{x\to x_0} \int_{\mathbb{R}^d} \big| v(x-y)-v(x_0-y)\big|^2 dy=0
		\end{equation*}
is equivalent to $v\ast v$ being continuous at $x_0$, which is true for $v\in L^2 (\R^d)$ by~\cite[Proposition 8.8]{Folland-99}. 

Thus we may assume that the left-hand sides of~\eqref{eq:CV mixed DM}-\eqref{eq:CV mixed DM bis} converge for any $v\in L^2 (\R^d)$. We now identify the limit $\sigma_0$ with the help of Theorem~\ref{thm:deF comp}. By density we may restrict to testing with a smooth compactly supported $v$ if needed, so that the multiplication operator $v(\cdot-y)$ is bounded on $\gh$. Theorem~\ref{thm:deF comp} implies that, along a subsequence, for any such $v$, $x_0 \in \R^d$ and $f\in L^2 (\R^d)$, 
\begin{equation}\label{eq:CV mixed 1}
\sqrt{\eps} \left\langle v (\cdot - x_0) \otimes \left( \cda (f) + c (f) \right)\right\rangle_{\Gamma_\eps} \to \int_{\gH}  \omega_u \left(v (\cdot - x_0)\right)  \left(\langle f | u \rangle + \langle u| f \rangle \right) d\mu (u).
\end{equation}
Introduce now a tiling $(Q^\eps_n)_{0\leq n \leq N}$ of $[0,R_\eps]^d$, say with squares of centers $x_n$ and vanishing side-length when $\eps \to 0$, where $R_\eps \to \infty$. We claim that, as operators, 
\begin{equation}\label{eq:Riemann}
 \sqrt{\eps}\left| \int_{\R^d} v(\cdot - y) \left( \cda_y + c_y\right) dy - \sum_n \left(\cda \left(\1_{Q^\eps_n}\right) + c \left(\1_{Q^\eps_n}\right)\right) v(\cdot-x_n) 
\right| \leq o_\eps (1) \left(\eps \cN +1\right). 
\end{equation}
Indeed 
a Cauchy-Schwarz inequality gives
\begin{multline*}
 \sqrt{\eps}\left| \int_{\R^d} v(\cdot - y) \left( \cda_y + c_y \right) dy - \sum_n \left(\cda \left(\1_{Q^\eps_n}\right) + c \left(\1_{Q^\eps_n}\right)\right) v(\cdot-x_n) 
\right| 
\\ = \left| \sum_n \int_{\R^d} \left(v(\cdot - y) - v (\cdot-x_n)\right) \left( \sqrt{\eps}\cda_y  + \sqrt{\eps}c_y\right) \1_{Q^\eps_n}(y)dy \right|
\\ \leq C \delta \eps \sum_n \int_{\R^d} \cda_y c_y \1_{Q^\eps_n} (y) dy + C\delta^{-1} \sum_n \int_{\R^d} \left(v(\cdot - y) - v (\cdot-x_n)\right)^2 \1_{Q^\eps_n} (y)\, dy\\
\leq C  \delta \eps \cN + \frac{C}{\delta} o_\eps (1)
\end{multline*}
using that 
\begin{equation*}
\sum_n \1_{Q^\eps_n} \equiv 1,
\end{equation*}
recognizing a Riemann sum and using that $v \in L^2 (\R^d).$ Choosing $\delta = \delta_\eps \to 0$ suitably slowly vindicates~\eqref{eq:Riemann}.

Next we obtain, after possibly a further extraction of subsequence
\begin{multline}\label{eq:CV mixed DM pre}
 \sqrt{\eps} \left\langle \int_{\R^d} v(\cdot - y) \left( \cda_y + c_y \right) dy \right\rangle_{\Gamma_{\eps}} - \sum_n \int_{\gH} \int_{\R^d} \omega_u (v(\cdot-x_n)) \left( u (y) + \overline{u(y)}\right) \1_{Q^\eps_n} (y) \,dy d\mu (u) \\ 
 \underset{\eps\to 0}{\to} 0. 
\end{multline}
This follows from~\eqref{eq:Riemann}, for each term 
\begin{equation*}\left(\cda \left(\1_{Q^\eps_n}\right) + c \left(\1_{Q^\eps_n}\right)\right) v(\cdot-x_n)\end{equation*}
is amenable to the use of~\eqref{eq:CV mixed 1}. With a suitable truncation of the sum in and a diagonal extraction we obtain convergence for each term and the sum along a common subsequence. Finally, the second term in the left-hand side of~\eqref{eq:CV mixed DM pre} equals the right-hand side of~\eqref{eq:CV mixed DM} by another Riemann sum argument, recalling that we may work with a smooth compactly supported $v$.
\end{proof}

\subsection{Proof of Theorem~\ref{thm:deF comp}}

 We recall the statement of~\cite[Theorem~4.2]{LewNamRou-14d}:

\begin{theorem}[\textbf{Grand-canonical quantum de Finetti theorem}]\label{thm:deF}\mbox{}\\
Consider a sequence $\eps \to 0$ of positive parameters, and associated sequence $\Gamma_\eps$ of states over $\gF(\gH)$ satisfying
\begin{equation*} \left\langle \left(\eps \cN \right)^\kappa \right\rangle_{\Gamma_{\eps}} < + \infty \end{equation*}
uniformly in $\eps$, for some $1 \leq \kappa$. 

There exists a unique probability measure $\mu \in \cP (\gH)$ such that, modulo the extraction of a subsequence,

\begin{enumerate}
 \item Expectations of anti-Wick observables converge and define the measure: for all $V \subset \gH$ a finite-dimensional subspace and $b\in C^0_c (V)$ we have 
 \begin{equation}\label{eq:AW CV}
  \left\langle b^{\rm aW}_{\eps} \right\rangle_{\Gamma_\eps} \to \int_{\mathfrak{H}} b(u) d\mu (u) 
 \end{equation}
\item Reduced density matrices converge
\begin{equation}\label{eq:CV DM}
\sqrt{k! \ell ! \eps ^k \eps^\ell}\Gamma_\eps ^{(k,\ell)} \underset{\eps\to 0}{\wto} \int_{\gH}  |u^{\otimes k} \rangle \langle u ^{\otimes \ell}| d\mu (u)
\end{equation}
weakly-$\star$ in the trace-class for all $k,\ell$ satisfying $\frac{k+\ell}{2} \leq \kappa$. In particular
\begin{equation*}
\sqrt{k! \ell ! \eps ^k \eps^\ell}\left\langle \cda (g_1) \ldots \cda (g_\ell) c (f_1) \ldots c(f_k) \right\rangle_{\Gamma_{\eps}} \underset{\eps\to 0}{\to} \int_{\gH}  \prod_{j=1} ^k \langle f_j | u \rangle \prod_{j=1}^\ell \langle u | g_j \rangle d\mu (u)
\end{equation*}
for all $f_1,\ldots,f_k,g_1,\ldots,g_\ell \in \gH$.
\end{enumerate}
\end{theorem}

Only the case $k=\ell$ of~\eqref{eq:CV DM} is worked out explicitly in~\cite{LewNamRou-14d}. The adaptation to $k\neq \ell$ is however straightforward, only the core calculations from e.g.~\cite[Lemma~4.2]{LewNamRou-14b} have to be adapted mutatis mutandis.
	
\begin{proof}[Proof of Theorem~\ref{thm:deF comp}]\mbox{}\\
\noindent \textbf{Step 1.} Let $C^0_c (\gH)$ denote continuous functions with compact support over $\gH$ and consider the algebra of observables 
\begin{equation*} \cA := \cB (\gh) \otimes C^0_c (\gH).\end{equation*}
Starting from $\Gamma_\eps \in \cS (\gh \otimes \gF (\gH))$ as in the theorem's statement we define a state $\Gammat_\eps\in \left( \cB (\gh) \otimes C^0_c (\gH) \right)^*$ over $\cA$ by testing it against a dense subset of elements of $\cA$. Namely, for any $A \in \cB(\gH)$, any finite-dimensional $V\subset \gH$ and $b\in C^0_c (V)$, we set 
\begin{equation*} \Gammat_\eps (A \otimes b) := \left\langle A \otimes b^{\rm aW}_{\eps} \right\rangle_{\Gamma_\eps}.\end{equation*}
That way $(\Gammat_\eps)_\eps$ is a bounded sequence of positive linear forms over $\cA$ (seen as a Banach space) and therefore it has a weak cluster point $\Gammat_0 \in \cA ^*$. Namely, along a subnet
\begin{equation}\label{eq:CV net 2}
\Gammat_\eps (C) \cvnet  \Gammat_0 (C) \mbox{ for all } C \in \cA.
\end{equation}
We now identify the cluster point, working along the just identified convergent subnet for the rest of the proof. 

For any positive operator $A\in \cB (\gH)$, we can define a (non-normalized) state $\Gamma ^A _\eps$ over $\gF (\gH)$ by setting 
\begin{equation*} 
\left\langle B \right\rangle_{\Gamma^A_\eps} = \left\langle A \otimes B\right\rangle_{\Gamma_\eps}.
\end{equation*}
Applying\footnote{Strictly speaking, we go back to the proof of Item (i) in~\cite{LewNamRou-14d} to identify any cluster point, not only sequential limits. This is done mutatis mutandis using Skorokhod's lemma~\cite{Skorokhod-74}.} Theorem~\ref{thm:deF} to $\Gamma ^A _\eps$ we find that there must exist a positive Borel measure $\mu_A$ on $\gH$ such that, along a further subnet, 
\begin{equation}\label{eq:CV net 3} 
\left\langle A \otimes b_{\eps}^{aW}\right\rangle _{\Gamma_\eps} \to \int_\mathfrak{H} b(u) d\mu_A (u).\end{equation}
As per~\eqref{eq:CV net 2}, $\Gammat_0$ is uniquely identified by
$$  \int_\mathfrak{H} b(u) d\mu_A (u) = \Gammat_0 (A\otimes b)$$
and there remains to further simplify the left-hand side.

Since (the operator norm is used below)  
\begin{equation*} \left\langle A \otimes b_{\eps}^{aW}\right\rangle _{\Gamma_\eps} \leq \left\Vert A \right\Vert \left\langle \1 \otimes b_{\eps}^{aW}\right\rangle _{\Gamma_\eps}\end{equation*}
for any positive function $b$ from a finite-dimensional subspace of $\gH$, we find that  
\begin{equation*} \int_\mathfrak{H} b(u) d\mu_A (u) \leq \left\Vert A \right\Vert \int_\mathfrak{H} b(u) d\mu_\1 (u).\end{equation*}
Picking any $V \subset \gH$ this implies that (approximating the characteristic function of $V$ by a sequence of continuous functions)
\begin{equation*} \mu_\1 (V) = 0 \Rightarrow \mu_A (V) = 0\end{equation*} for any positive
bounded operator $A\in \cB (\gh)$. By Radon-Nykodym's theorem, we deduce that
for any positive bounded $A$, there exists a map $u\mapsto\omega_u (A) \in L^1 (\gH,
d\mu_\1)$ such that
\begin{equation*} 
\int_\mathfrak{H} b(u) d\mu_A (u) = \int_\mathfrak{H} b(u) \omega_u(A) d\mu_\1 (u).
\end{equation*}
Upon redefining $\omega_u(A)$ if necessary we can assume $\mu_\1$ is a probability. From the definition it also follows that $\omega_u(A)$ is $\mu_\1$ almost-surely a bounded linear function of $A$.

Next we can split a general bounded operator in the form 
\begin{equation}\label{eq:split op}
 A = A_{r+} - A_{r-} + \im A_{i+} - \im A_{i-}   
\end{equation}
with four positive operators $A_{r+},A_{r-},A_{i+},A_{i-}$. Applying the above to each term separately we find a $\mu:= \mu_\1 \in \cP(\gH)$ and $u\mapsto \omega_u (.)$ a $L^1 (\gH,d\mu)$ map from $\gH$ to the state-space of $\gh$ such that
\begin{equation*} \left\langle A \otimes b_{\eps}^{aW}\right\rangle _{\Gamma_\eps} \to \int_\mathfrak{H} \omega_u(A) d\mu (u)\end{equation*}
for any $A\in \cB(\gH)$ and any $b\in C^0_c (V)$ with $V$ a finite-dimensional subspace of $\gH$. This is the first statement of the theorem.

\medskip

\noindent \textbf{Step 2.} Under our assumptions, $\sqrt{k! \ell ! \eps ^k \eps^\ell}\Gamma_\eps ^{(k,\ell)} (A)$ is a bounded sequence of trace-class operators for any positive bounded $A$ and $\frac{k+\ell}{2}\leq \kappa$. Hence we may extract a weak-$\star$ convergent subsequence. If $A$ varies in a separabable subspace of the bounded operators (e.g. the span of compact operators and the identity, as in the theorem's statement), we can use a dense countable subset thereof to obtain convergence along a common subsequence for all $A$, modulo a diagonal extraction argument. 

To obtain the second statement of the theorem we further extract a subnet along which Item (1) holds.  There remains to apply~\eqref{eq:CV DM} to each (non-normalized)  state $\Gamma ^A _\eps$ defined above, with $A$ a positive operator, and use the splitting~\eqref{eq:split op} to generalize to all operators in the statement. The measure in~\eqref{eq:CV DM} being the same as that in~\eqref{eq:AW CV}, we identify the limit in~\eqref{eq:CV DM comp} by using~\eqref{eq:AW CV comp}.
\end{proof}

Under more restrictive assumptions we may ensure that $\omega_u (\,.\,)$ is almost surely a normal state, i.e. can be represented by a density matrix:

\begin{corollary}[\textbf{Quantum de Finetti for composite localized states}]
  \label{cor:1}\mbox{}\\
  Suppose that, in addition to the assumptions of Theorem~\ref{thm:deF comp},
  the sequence $\Gamma_{\varepsilon}$ satisfies the bound
  \begin{equation}\label{eq:trapped}
    \left\langle L\otimes \1  \right\rangle_{\Gamma_{\varepsilon}}<+\infty\;,
  \end{equation}
  uniformly in $\varepsilon$, for some positive operator $L$
  on $\mathfrak{h}$ with compact resolvent. Then the $\mu$-measurable map $\omega$ of
  Theorem~\ref{thm:deF comp} takes values in the set of normal states on $\mathfrak{h}$,
  \emph{i.e.} in the set of positive, normalized, trace-class operators.
\end{corollary}

This result will not be used in our proof of Theorem~\ref{thm:main}. The a priori absence of a suitable operator $L$ ensuring the validity of~\eqref{eq:trapped} precisely reflects the lack of trapping of the operator acting on the particle.

\begin{proof}
  On the one hand, we know that by Theorem~\ref{thm:deF comp}, there exists a
  measure $\mu$ and a state-valued map $\omega$ such that the expectation of
  anti-Wick observables converges. In addition, as illustrated in the proof
  of the theorem, such convergence identifies the couple $(\mu,\omega)$.

  Now, let us define the family of states
  \begin{equation*}
  \Gamma_{\varepsilon}^{(L)}:=(L+1)^{1/2}\Gamma_{\varepsilon}(L+1)^{1/2}.
  \end{equation*}
  For $\Gamma_{\varepsilon}^{(L)}$ we proceed as in the proof of Theorem~\ref{thm:deF comp}, substituting the
  algebra of observables $\mathcal{A}$ with
  \begin{equation*}
    \mathcal{K}:= \mathcal{L}^{\infty}(\mathfrak{h})\otimes C_c^0(\mathfrak{H})\;,
  \end{equation*}
  where $\mathcal{L}^{\infty}(\mathfrak{h})$ is the space of compact operators. Thanks to this
  modification, we can identify a limit measure $\mu^{(L)}$, and a
  $\mu^{(L)}$-measurable map
  \begin{equation*} \omega^{(L)} : \begin{cases}
             \gH \to \mathcal{L}^1_{+,1} (\gh)\\
             u \mapsto \omega^{(L)}_u
            \end{cases}
            \end{equation*}
            where $\mathcal{L}^1_{+,1} (\gh)$ is the set of normal states on $\mathfrak{h}$, dual
            to the set of compact operators. The drawback is that in this
            case $\mu^{(L)}$ \emph{could fail to be a probability measure}. Let us
            remark that one could identify different limit measures along
            different subsubnets of the one used to obtain $(\mu,\omega)$ from
            $\Gamma_{\varepsilon}$.

            Now, let us fix $A\in \mathcal{B}(\mathfrak{h})$, $V\subset \mathfrak{H}$ finite dimensional, and $b\in
            C^0_c(V)$. On the one hand, by Theorem~\ref{thm:deF comp}, along the
            subsubnet
            \begin{equation*}
              \langle  A\otimes b^{\mathrm{aW}}_{\varepsilon}  \rangle_{\Gamma_\varepsilon}\to \int_{ \mathfrak{H}}^{}\omega_u(A)b(u)  d\mu(u)\;,
            \end{equation*}
            and on the other hand,
            \begin{equation*}
              \langle  (L+1)^{-1/2}A(L+1)^{-1/2}\otimes b^{\mathrm{aW}}_{\varepsilon}  \rangle_{\Gamma^{(L)}_\varepsilon}\to \int_{ \mathfrak{H}}^{}\omega_u^{(L)}\bigl((L+1)^{-1/2}A(L+1)^{-1/2}\bigr)b(u) d\mu^{(L)}(u)\;,
            \end{equation*}
            since $(L+1)^{-1/2}A(L+1)^{-1/2}\in \mathcal{L}^{\infty}(\mathfrak{h})$ for any $A\in
            \mathcal{B}(\mathfrak{h})$. However,
            \begin{equation*}
              \langle  A\otimes b^{\mathrm{aW}}_{\varepsilon}  \rangle_{\Gamma_\varepsilon}=\langle  (L+1)^{-1/2}A(L+1)^{-1/2}\otimes b^{\mathrm{aW}}_{\varepsilon}  \rangle_{\Gamma^{(L)}_\varepsilon}
            \end{equation*}
            by  definition of $\Gamma_{\varepsilon}^{(L)}$, and thus
            \begin{equation*}
              \int_{ \mathfrak{H}}^{}\omega_u(A)b(u)  d\mu(u)=\int_{ \mathfrak{H}}^{}\omega_u^{(L)}\bigl((L+1)^{-1/2}A(L+1)^{-1/2}\bigr)b(u) d\mu^{(L)}(u).
            \end{equation*}
            Fixing a bounded $A$ and varying $b$ implies that 
            $$\omega_u(A)  d\mu(u) = \omega_u^{(L)}\bigl((L+1)^{-1/2}A(L+1)^{-1/2}\bigr)d\mu^{(L)}(u).$$
            Hence in particular with $A= \1$ we find $\mu= \mu_L$. Also, $\mu$-almost surely, 
            \begin{equation*}\omega_u(\,\cdot\,)=\omega^{(L)}_u\bigl((L+1)^{-1/2}\,\cdot \,(L+1)^{-1/2}\bigr).\end{equation*}
            The limit is the same along any subsubnet, and thus it holds on the original
            subnet as well. Therefore, it follows that $\omega_u\in \mathcal{L}^1_{+,1}(\mathfrak{h})$.
  \end{proof}

  \section{Convergence of the energy}\label{sec:energy}

For completeness we now revisit the proof of 

\begin{theorem}[\textbf{Energy convergence}]\label{thm:ener}\mbox{}\\
With the assumptions and notation of Section~\ref{sec:main}, we have that
\begin{equation*} E_{(\alpha)} ^{(V)} \underset{\alpha \to \infty}{\to} E^{(V)}_\mathrm{Pek}(1).\end{equation*}
In particular this holds true with the external potential $V\equiv 0$.
\end{theorem}

Our proof is in the spirit of~\cite{CorFalOli-20}, but we use quasi-classical measures as constructed in the previous section, leading to mild simplifications. In view of Theorem~\ref{thm:deF comp}, the natural limit energy takes general abstract states as arguments. We discuss this first in a subsection, and prove that this does not lower the energy as compared to what was defined in Section~\ref{sec:main}. We will complete the proof of Theorem~\ref{thm:ener} in a second subsection.

\subsection{Generalized Pekar energies}

Let 
\begin{equation*}h:=-\Delta + V\end{equation*}
and $W$ be as in~\eqref{eq:eff pot}, and identified with the multiplication operator by $W(x-y)$ on the two-particle space $L^2_{\rm part} \otimes L^2_{\rm part}$. 

We start this discussion by generalizing Pekar's energy functional to take mixed states as arguments:

\begin{lemma}[\textbf{Mixed Pekar functional}]\label{lem:mixed pek}\mbox{}\\
Any minimizer of   
 \begin{equation*}
 \gamma \mapsto \Tr\left[ h \gamma \right] + \Tr \left[ W(x-y) \gamma \otimes \gamma \right]
 \end{equation*}
 amongst positive trace-class operators of trace $1$ must be rank one. Hence any minimizer is of the form $\gamma = |\psi \rangle \langle \psi|$ with $\psi$ a minimizer for~\eqref{eq:pekar_energy} with $m=1$.
\end{lemma}

Similar arguments may be found e.g. in~\cite[Section~5]{Seiringer-02} or~\cite[Section~2]{BauSei-01}.

\begin{proof}
The existence of minimizers follows by a concentration-compactness argument similar to that leading to the existence for~\eqref{eq:pekar_energy}. We skip details and denote $\gamma_0$ a minimizer. 

Consider a variation 
\begin{equation*} 
\gamma = (1- \eps)\gamma_0 + \eps \sigma
\end{equation*}
with $0<\eps <1$ and $\sigma$ a positive trace-class operator of trace $1$. We must have 
$$ \Tr\left[ h \gamma_0 \right] + \Tr \left[ W(x-y) \gamma_0 \otimes \gamma_0 \right] \leq \Tr\left[ h \sigma \right] + \Tr \left[ W(x-y) \sigma \otimes \sigma \right].$$
Note that 
$$ \Tr \left[ W(x-y) \sigma \otimes \sigma \right] = \iint_{\R^d \times  \R^d} W(x-y) \sigma (x,x) \sigma (y,y)dxdy.$$
Expanding and taking $\eps$ small enough we find that necessarily (keeping only the $O(\eps)$ term in the expansion)
\begin{equation*} 
\Tr\left[  \gamma_0 \left( h + W\ast \rho_{\gamma_0}\right) \right] \leq \Tr\left[  \sigma \left( h + W\ast \rho_{\gamma_0}\right) \right]\end{equation*}
where $\rho_{\gamma_0} (x) = \gamma_0 (x,x)$ is the density of $\gamma_0$.

Hence $\gamma_0$ must also minimize the linearized 
\begin{equation*} \gamma \mapsto \Tr\left[  \gamma \left( h + W\ast \rho_{\gamma_0}\right) \right],\end{equation*}
which in particular shows that the Schr\"odinger operator $h + W\ast \rho_{\gamma_0}$ has at least a ground energy state. Then, $\gamma$ must have its image in the ground energy space of $h + W\ast \rho_{\gamma_0}$, but the latter has dimension one by well-known arguments (see e.g.~\cite[Theorem~2.3]{Rougerie-EMS} or~\cite[Section~XIII.12]{ReeSim4}). 
\end{proof}

We now turn to a functional taking generalized states as arguments. For an abstract state $\omega$ on  $L^2_{\rm part} (\R^d)$  (a positive linear functional over bounded operators acting on $L^2_{\rm part} (\R^d)$) let, in analogy with~\eqref{eq:pekar_functional},  
\begin{equation*} 
\mathcal{E}_\mathrm{Pek}^{(V)}(\omega) := \omega (h) + \omega \otimes \omega \left(W(x-y)\right) = \omega \otimes \omega \left( \frac{h_x + h_y}{2} + W (x-y)\right) 
\end{equation*}
and ($\mathrm{Gen}$ for generalized)
\begin{equation}\label{eq:gen ener}
 E_\mathrm{Gen}^{(V)}:= \inf\left\lbrace \mathcal{E}_\mathrm{Pek}^{(V)}(\omega), \, \omega \in \cS \left( L^2 (\R^d)\right) \mbox{ as defined in~\eqref{eq:states}}\right\rbrace.
\end{equation}
Implicit in the above is the fact that the minimization is performed under the constraint that 
$$ A \mapsto \omega \left( h^{1/2} A h^{1/2} \right)$$
is a bounded linear map over bounded operators $A$ acting on $L^2_{\rm part} (\R^d),$ so that $\omega (h)$ makes sense (we use that $h$ is a non-negative operator here). Under our assumptions one easily proves that
\begin{equation}\label{eq:two part op} 
 H_2:= \frac{h_x + h_y}{2} + W (x-y) \geq - C  
\end{equation}
for some constant $C$, and hence the infimum above is well-defined. We have the 

\begin{lemma}[\textbf{Generalized energy = Pekar energy}]\label{lem:gen ener}\mbox{}\\
 With the previous definitions
 \begin{equation*} E_\mathrm{Gen}^{(V)} = E^{(V)}_\mathrm{Pek}(1).\end{equation*}
\end{lemma}

\begin{proof}
In view of Assumption~\ref{assum:V}, we may for this proof assume without loss that $h\geq 0$ as an operator.

Denote $\gh = L^2 (\R^d)$ for brevity. We have~\cite[Chapter~4]{Schatten-60} that the dual of bounded operators $\left(\cB (\gh)\right)^*$ is the bidual of the trace-class $\cL^1 (\gh)$. Hence, by Goldstine's theorem\footnote{If $X$ is a Banach space, its unit ball is dense in that of the bidual $X^{**}$ for the weak-$\star$ topology, see e.g.~\cite[Exercise~1 on page 128]{Rudin2}.}, for any abstract state $\omega$ there exists a net of positive trace-class operators $\gamma_n$ such that 
\begin{equation}\label{eq:density}
 \Tr \left[ \gamma_n B \right] \cvnet \omega (B) 
\end{equation}
for any bounded operator $B$. The rest of the proof is then aking to that of~\cite[Proposition~2.8]{CorFalOli-20}.

For a state $\omega$ with $\mathcal{E}_\mathrm{Gen}^{(V)}(\omega)<\infty$ it follows from~\eqref{eq:two part op} that $\omega (h) < \infty$ and $ \omega \otimes \omega \left( H_2 \right) < \infty.$
Hence 
\begin{equation*} \omega ^h (B) := \omega \left( h^{1/2} B h^{1/2}\right)\end{equation*}
defines a positive linear functional on $\cB (\gh)$ as well, to which we may apply the above, obtaining a net of trace-class operators $\gamma_n^h$ such that 
\begin{equation}\label{eq:density 2}
 \Tr \left[ \gamma_n^h B \right] \cvnet \omega^h (B). 
\end{equation}
Applying~\eqref{eq:density} directly to $\omega$ yields another net $\gamma_n$, but testing~\eqref{eq:density 2} with $B$ of the form $h^{-1/2} \widetilde{B} h^{-1/2}$ for a bounded $\widetilde{B}$ shows that one can take
\begin{equation*} \gamma_n = h^{-1/2} \gamma_n ^h h ^{-1/2}.\end{equation*}
Similarly 
\begin{equation*} \omega_2 (B_2) := \omega \otimes \omega \left(  \sqrt{H_2 +C} B_2 \sqrt{H_2 +C} \right)\end{equation*}
defines a positive linear functional on bounded operators $B_2$ on $\gh^{\otimes 2}$, where $C$ is a constant such that $H_2 + C \geq 0$. We deduce that $\omega_2$ is the limit of a net of trace-class operators that we may identify to
\begin{equation*} \left( h^{-1/2} \gamma_n ^h h ^{-1/2}\right) \otimes \left( h^{-1/2} \gamma_n ^h h ^{-1/2}\right) \end{equation*}
as above. Using~\eqref{eq:density 2} and the fact that $W$ is a bounded mutliplication operator, we conclude that for any state $\omega$ with $\mathcal{E}_\mathrm{Gen}^{(V)}(\omega)<\infty$, there exists a net $(\gamma_n)$ of trace-class operators such that 
\begin{equation*}\mathcal{E}_\mathrm{Pek}^{(V)}(\gamma_n) \cvnet \mathcal{E}_\mathrm{Pek}^{(V)}(\omega).\end{equation*}
This leads to
\begin{equation}\label{eq:gen ener 2}
 E_\mathrm{Gen}^{(V)}\geq \inf\left\lbrace \mathcal{E}_\mathrm{Pek}^{(V)}(\gamma), \gamma \in \cL^1(\gh), \gamma \geq 0, \Tr \gamma = 1\right\rbrace.
\end{equation}
The opposite inequality follows from the variational principle. The right-hand side of the above is the Pekar energy~\eqref{eq:pekar_functional} generalized to a mixed state 
\begin{equation*} \gamma = \sum_{j\geq 1} \lambda_j |u_j\rangle \langle u_j|\end{equation*}
with $\lambda_j \geq 0, \sum_j \lambda_j = 1$ and an orthonormal basis $(u_j)_j$ of $\gh$. 
We hence conclude from Lemma~\ref{lem:mixed pek} that   
 \begin{equation*} E_\mathrm{Gen}^{(V)} = E^{(V)}_\mathrm{Pek}(1)\end{equation*}
 as desired.
\end{proof}

\subsection{Proof of Theorem~\ref{thm:ener}}

Again, without loss of generality (i.e. adding a constant if needed), we
assume that $h\geq 0$. We consider a sequence of quasi-minimizers as
in~\eqref{eq:minimizing_sequence}. Under our assumptions, applying the
Cauchy-Schwarz inequality to the interaction term immediately leads to the a
priori bound
\begin{equation}\label{eq:apriori}
\left\langle\Psi_\alpha| \left( - \Delta + V \right) \otimes \1 + \1 \otimes \cN_\alpha | \Psi_\alpha \right\rangle \leq C 
\end{equation}
independently of $\alpha$. Here $\cN_\alpha$ is the scaled particle number~\eqref{eq:scaled number}. We apply Theorem~\ref{thm:deF comp} with $\kappa = 1, \eps = \alpha^{-2}$, obtaining a probability measure $\mu$ over $L^2 (\R^d)$ and a $\mu$-measurable map $\omega_u$ from $L^2 (\R^d)$ to the state-space of $L^2 (\R^d)$. Combining with Corollary~\ref{cor:deFint} we may pass to the limit in the interaction term. For the field energy, we pass to the liminf using~\eqref{eq:CV DM comp} with $A=\1$, $k=\ell = 0$ and the fact that the trace-norm is lower semi-continuous under weak-$\star$ convergence in the trace-class.

As regards the particle energy we denote $\gamma_\alpha$ the particle reduced density matrix of $|\Psi_\alpha \rangle \langle \Psi_\alpha|$. Since $\Tr \left( \gamma_\alpha h\right ) < \infty$, we have that 
\begin{equation*} \gamma_\alpha^h (B) := \Tr \left( h^{1/2} \gamma_\alpha h^{1/2} B\right) \end{equation*}
defines a bounded sequence of positive linear forms over bounded operators. Extracting a further weakly-$\star$ convergent subnet and identifying the limit by testing with $B$ of the form $h^{-1/2}\tilde{B}h^{-1/2}$ we deduce that 
\begin{equation*} \Tr\left( h \gamma_\alpha \right) \cvnet \int \omega_u (h) d\mu (u).\end{equation*}
All in all
\begin{align*}
 \liminf_{\alpha \to \infty} E_{(\alpha)} ^{(V)} &\geq \int_{L^2 (\R^d)}  \left( \omega_u (h) + \left\Vert u \right\Vert_{L^2} ^2 + \int_{\R^d} \omega_u (v(\cdot-z)) \left( u (z) + \overline{u(z)}\right) dz  \right) d\mu (u)\\
 &\geq \inf \left\{ \omega (h) + \left\Vert u \right\Vert_{L^2} ^2 + \int_{\R^d} \omega (v(\cdot-z)) \left( u (z) + \overline{u(z)}\right) dz, u \in L^2 (\R^d), \omega \in \cS (L^2 (\R^d)) \right\}
\end{align*}
since $\mu$ is a probability measure. Minimizing with respect to $u$ at fixed $\omega$ in a similar manner as in~\eqref{eq:pekar_functional} leads to a real-valued $u$ such that 
\begin{equation*} u(z) = - \omega \left(v(.-z)\right)\end{equation*}
and an energy 
\begin{align*}
 \omega (h) - \left\Vert u \right\Vert_{L^2} ^2 &= \omega (h) - \iint \omega (v(\cdot-y)) \omega (v(\cdot-z)) dy dz\\
 &= \omega (h) - \iint \omega_{x_1} \otimes \omega_{x_2} \left( v(x_1-y)  v(x_2-z)\right) dy dz\\
 &= \mathcal{E}_\mathrm{Pek}^{(V)}(\omega)
\end{align*}
where we inverted the integral over $y,z$ and the expectation in $\omega \otimes \omega$ in the last step, recalling~\eqref{eq:eff pot}. We conclude that 
\begin{equation*} \liminf_{\alpha \to \infty} E_{(\alpha)} ^{(V)} \geq E_\mathrm{Gen}^{(V)}.\end{equation*}
There remains to use Lemma~\ref{lem:gen ener} and recall that the upper bound 
\begin{equation*} E_\mathrm{Pek}^{(V)} (1) \geq E_{(\alpha)} ^{(V)}\end{equation*}
follows from the trial state argument sketched in Section~\ref{sec:main}.

\section{Convergence of states, proof of Theorem~\ref{thm:main}}\label{sec:limit}
	
Since our main result Theorem~\ref{thm:main} is stated modulo subsequence, we take the liberty of not indicating all extractions of subsequences/subnets in the arguments of this section. 

We start from a sequence of states 
\begin{equation*} \Gamma_\alpha = |\Psi_\alpha\rangle \langle \Psi_\alpha |\end{equation*}
as in the statement of the theorem. As in the previous section we have that 
\begin{equation}\label{eq:apriori bis}
\left\langle\Psi_\alpha| \left( - \Delta + V \right) \otimes \1 + \1 \otimes \cN_\alpha | \Psi_\alpha \right\rangle \leq C 
\end{equation}
and we may apply Theorem~\ref{thm:deF comp} with $\kappa = 1$, obtaining a probability measure $\mu$ over $L^2 (\R^d)$ and a state-valued map $\omega_u$. Let $\gamma_\alpha$ be the particle density matrix of $\Gamma_\alpha$, as in Section~\ref{sec:DMs}. We may extract a weak-$\star$ convergent subsequence in the trace-class:
\begin{equation}\label{eq:weak star}
 \Tr \left[ \gamma_\alpha K \right] \underset{\alpha \to \infty}{\to} \Tr \left[ \gamma_\infty K \right]  
\end{equation}
for any compact operator $K$ over $L^2(\R^d)$. Identifying the limit using Theorem~\ref{thm:deF comp}, it must be that 
\begin{equation}\label{eq:gamma inf}
 \gamma_\infty = \int \omega_u ^{\nor} d\mu (u) 
\end{equation}
with $\omega_u ^{\nor}$ the normal part of $\omega_u$, i.e. the unique trace-class operator satisfying
\begin{equation}\label{eq:normal part}
 \omega_u (K) = \Tr \left[\omega_u ^{\nor} K \right] \mbox{ for any compact operator } K. 
\end{equation}

Arguing in a similar manner for the field density matrix 
\begin{equation*}\gamma^f :=\Gamma_{\alpha}^{(1,1)}\end{equation*}
we find 
\begin{equation*} 
\gamma^f \underset{\alpha \to \infty}{\overset{\star}{\wto}} \int |u\rangle \langle u| d\mu (u)
\end{equation*}
along a subsequence, instead of just a subnet as in Theorem~\ref{thm:deF comp}. As regards the particle-field density matrix $\sigma_\alpha$, we consider $\sigma_{\alpha}(x,x;z)$ as a $L^1_x L^2_z$ function as in the proof of Corollary~\ref{cor:deFint} and deduce 
\begin{equation*}
\iint_{\R^2 \times \R^2} \sigma_{\alpha} (x,x;z) v(x-z) dx dz \underset{\alpha\to \infty}{\to} \int_{\gH} \int_{\R^d} \omega_u (v(\cdot-z)) \left( u (z) + \overline{u(z)}\right) dz d\mu (u).
\end{equation*}
Now, we aim at turning the weak convergences from~\eqref{eq:weak star} and Theorem~\ref{thm:deF comp} into strong ones. For that we prove that no mass is lost in the limit:

\begin{lemma}[\textbf{No loss of mass}]\label{lem:mass}\mbox{}\\
Let $\gamma_\infty$ be the weak-$\star$ limit of the particle density matrix, introduced above. We have that 
\begin{equation*} \Tr \left[ \gamma_\infty \right]= 1\end{equation*}
and hence
\begin{equation*} \gamma_\alpha \underset{\alpha \to \infty}{\to}\gamma_\infty\end{equation*}
along a subsequence, strongly in trace-class norm.
\end{lemma}

\begin{proof}
That the first statement implies the second is classical~\cite{dellAntonio-67,Simon-79}. We thus focus on the mass of the limit density matrix. 

\medskip 

\noindent\textbf{Step 1.} Let $\chi_R$ be a localization function as in Section~\ref{sec:loc}. We claim that 
\begin{equation}\label{eq:loc mass}
\Tr \left[ \chi_R \gamma_\alpha \chi_R \right] \underset{\alpha \to \infty}{\to} \Tr \left[ \chi_R \gamma_\infty \chi_R \right]. 
\end{equation}
Indeed, let $\gamma_\alpha ^k$ be the positive  operator
\begin{equation*} \gamma_\alpha ^k = (1-\Delta)^{1/2} \gamma_\alpha (1-\Delta)^{1/2}.\end{equation*}
It follows from~\eqref{eq:apriori bis} that $\gamma_\alpha ^k$ is uniformly bounded in trace-class norm. Thus, modulo a possible further extraction 
\begin{equation*} 
\gamma_\alpha ^k \underset{\alpha \to \infty}{\overset{\star}{\wto}} \gamma_\infty ^k = (1-\Delta)^{1/2} \gamma_\infty (1-\Delta)^{1/2} 
\end{equation*}
weakly-star in the trace-class, where we identified the limit by testing against $(1-\Delta)^{-1/2} K (1-\Delta)^{-1/2}$ for a compact operator $K$. Then 
\begin{align*}
\Tr \left[ \chi_R \gamma_\alpha \chi_R \right] &= \Tr \left[ \chi_R (1-\Delta)^{-1/2} \gamma_\alpha ^k  (1-\Delta)^{-1/2} \chi_R\right] \\
&\underset{\alpha \to \infty}{\to} \Tr \left[ \chi_R (1-\Delta)^{-1/2} \gamma_\infty ^k  (1-\Delta)^{-1/2} \chi_R\right] \\
&= \Tr \left[ \chi_R \gamma_\infty \chi_R \right]
\end{align*}
because $\chi_R (1-\Delta)^{-1/2}$ is compact. Indeed, since $\chi_R$ is smooth with compact support it is in any $L^p$ space, while $(1-\Delta)^{-1/2}$ acts in Fourier variables as the multiplyier by $\left(1+|k|^2\right)^{-1/2}$, which belongs to $L^q$ for $q>d$. Hence, the Kato-Seiler-Simon inequality~\cite[Chapter~4]{Simon-79} implies that $\chi_R (1-\Delta)^{-1/2}$ is in the Schatten space $\cL^q$ for any $q>d$. 
 
\medskip 

\noindent\textbf{Step 2.} We now prove that 
\begin{equation}\label{eq:lim mass}
\lim_{R\to \infty}\lim_{\alpha \to \infty}  \Tr \left[ \chi_R \gamma_\alpha \chi_R \right] = 1. 
\end{equation}
Let $\chi_R$ be as above and 
\begin{equation*} \eta_R = \sqrt{1-\chi_R^2}.\end{equation*}
Then~\eqref{eq:full_energy_localization} implies (bounding $H_\alpha^{(0)}$ from below by its' lowest eigenvalue)
\begin{equation*}
				\liminf_{\alpha\to\infty}\, \mathrm{Tr}_{\mathfrak{H}}\big(H_\alpha^{(V)}\, \Gamma_\alpha\big)
				\ge\; \liminf_{R\to\infty} \liminf_{\alpha\to\infty} \left( \mathrm{Tr}_{\mathfrak{H}}\big(H_\alpha^{(V)}\, \Gamma_{\alpha,\chi_R}\big) +\mathrm{Tr}_{\mathfrak{H}} \big(\Gamma_{\alpha,\eta_R}\big) E^{(0)}_\alpha\right).
\end{equation*}
with $\Gamma_{\alpha,\chi_R}$ and $ \Gamma_{\alpha,\eta_R}$ the $\chi_R-$ and $\eta_R-$localized states constructed from $\Gamma_\alpha$. 

Next, combining with the energy upper bound obtained as sketched in Section~\ref{sec:main},
\begin{equation*} E_\mathrm{Pek}^{(V)}(1) \geq \liminf_{R\to\infty} \liminf_{\alpha\to\infty} \left(\mathrm{Tr}_{\mathfrak{H}} \big(\Gamma_{\alpha,\chi_R}\big) E^{(V)}_\alpha + \mathrm{Tr}_{\mathfrak{H}} \big(\Gamma_{\alpha,\eta_R}\big) E^{(0)}_\alpha \right).\end{equation*}
Inserting the energy convergence from Theorem~\ref{thm:ener} and using~\eqref{eq:localized_traces} leads to  
\begin{equation*} E_\mathrm{Pek}^{(V)}(1) \geq \liminf_{R\to\infty} \liminf_{\alpha\to\infty} \left(E_\mathrm{Pek}^{(V)}(1)\mathrm{Tr}_{\mathfrak{H}} \big(\Gamma_{\alpha,\chi_R}\big) + E_\mathrm{Pek}^{(0)}(1)\left(1-\mathrm{Tr}_{\mathfrak{H}} \big(\Gamma_{\alpha,\chi_R}\big)\right)\right)\end{equation*} 
so that  
\begin{equation*} 0 \geq \liminf_{R\to\infty} \liminf_{\alpha\to\infty}\left(E_\mathrm{Pek}^{(0)}(1) - E_\mathrm{Pek}^{(V)}(1)\right) \left(1-\mathrm{Tr}_{\mathfrak{H}} \big(\Gamma_{\alpha,\chi_R}\big)\right).\end{equation*}
But 
\begin{equation*} E_\mathrm{Pek}^{(V)}(1) < E_\mathrm{Pek}^{(0)}(1)\end{equation*}
since $V < 0$, as follows by using a translation-invariant ground state as trial state for the functional with trapping potential. It must thus be that 
\begin{equation*} \liminf_{R\to\infty} \liminf_{\alpha\to\infty} \mathrm{Tr}_{\mathfrak{H}} \big(\Gamma_{\alpha,\chi_R}\big) = 1,\end{equation*}
which implies~\eqref{eq:lim mass}, using~\eqref{eq:localized_density_particle}.

\medskip 

\noindent\textbf{Conclusion.} Combining~\eqref{eq:loc mass} with~\eqref{eq:lim mass} leads to 
\begin{equation*} \lim_{R\to \infty}  \Tr \left[ \chi_R \gamma_\infty \chi_R \right] = 1 \end{equation*}
and the result follows.
\end{proof}

Combining the lemma with~\eqref{eq:gamma inf} implies that 
\begin{equation*} \Tr \left[\omega_u ^{\nor}\right] = 1\end{equation*}
for $\mu$-almost every $u$, where the normal part is defined as in~\eqref{eq:normal part}. Hence $\omega_u$ coincides $\mu$-almost surely with its normal part, a positive trace-class operator. We denote the latter $\gamma_u$, which has trace~$1$.

We may now return to Theorem~\ref{thm:deF comp} and pass to the limit in the energy as in Section~\ref{sec:energy} to obtain  
\begin{align*}
E_\mathrm{Pek}^{(V)}(1) \geq \liminf_{\alpha \to \infty} E_{(\alpha)} ^{(V)} &\geq \int_{L^2 (\R^d)}  \left( \Tr\left[ h \gamma_u\right] + \left\Vert u \right\Vert_{L^2} ^2 + \int_{\R^d} \Tr\left[ \gamma_u (v(\cdot-z))\right] \left( u (z) + \overline{u(z)}\right) dz  \right) d\mu (u)\\
&\geq \int_{L^2 (\R^d)}  \left( \Tr\left[ \left(h  + W\ast \rho_{\gamma_u}\right)\gamma_u\right] \right) d\mu (u)\\
&\geq E_\mathrm{Pek}^{(V)}(1).
\end{align*}
To go to the second line we have minimized with respect to $u$, obtaining 
\begin{equation}\label{eq:min u}
 u(z) = - \Tr \left[ \gamma_u v(.-z)\right] = - \int_{\R^d} \rho_{\gamma_u} (x) v(x-z) dx
\end{equation}
with $\rho_{\gamma_u} (x) = \gamma_u (x,x)$ the density of $\gamma_u.$
To go to the third line we used Lemma~\ref{lem:mixed pek}, i.e. that the Pekar functional for mixed states leads to the same minimization problem as the usual one. This fact and the previous chain of inequalities (there must be equality throughout) also imply that for $\mu$-almost every $u$, 
\begin{equation*}\gamma_u = |\psi \rangle \langle \psi |\end{equation*}
with $\psi$ a minimizer of the Pekar energy functional~\eqref{eq:pekar_functional} at mass $1$. We also must have~\eqref{eq:min u} and hence 
\begin{equation*}u= u_\psi= - v \ast |\psi|^2\end{equation*} 
as in~\eqref{eq:fieldconfig} for $\mu$-almost every $u$. Theorem~\ref{thm:main} follows upon defining 
\begin{equation*} dP(\psi) := \int \1_{u=u_\psi} d\mu(u).\end{equation*}

\newpage


%

\end{document}